\documentclass[12pt]{amsart}
\usepackage{}

\usepackage{amsmath}
\usepackage{amsfonts}
\usepackage{amssymb}
\usepackage[all]{xy}           

\usepackage{bbding}
\usepackage{txfonts}
\usepackage{amscd}

\usepackage[shortlabels]{enumitem}
\usepackage{ifpdf}
\ifpdf
  \usepackage[colorlinks,final,backref=page,hyperindex]{hyperref}
\else
  \usepackage[colorlinks,final,backref=page,hyperindex,hypertex]{hyperref}
\fi
\usepackage{tikz}
\usepackage[active]{srcltx}

\makeatletter

\newtheorem{thm}{Theorem}[section]
\newtheorem{lem}[thm]{Lemma}
\newtheorem{cor}[thm]{Corollary}
\newtheorem{pro}[thm]{Proposition}
\newtheorem{ex}[thm]{Example}
\newtheorem{rmk}[thm]{Remark}
\newtheorem{defi}[thm]{Definition}

\newcommand {\emptycomment}[1]{}

\newcommand{\liu}[1]{\textcolor{black}{ #1}}

\newcommand{\lon }{\,\rightarrow\,}
\newcommand{\be }{\begin{equation}}
\newcommand{\ee }{\end{equation}}

\newcommand{\g}{\mathfrak g}

\newcommand{\Real}{\mathbb R}

\newcommand{\huaL}{\mathcal{L}}
\newcommand{\huaR}{\mathcal{R}}

\newcommand{\huaG}{\mathcal{G}}

\newcommand{\huaC}{{\mathcal{C}}}

\newcommand{\huaH}{\mathcal{H}}

\newcommand{\huaO}{{\mathcal{O}}}

\newcommand{\frks}{\mathfrak s}

\newcommand{\frkB}{\mathfrak B}

\newcommand{\frkG}{\mathfrak G}
\newcommand{\frkH}{\mathfrak H}

\newcommand{\frkL}{\mathfrak L}

\newcommand{\frkR}{\mathfrak R}

\newcommand{\half}{\frac{1}{2}}
\newcommand{\Deg}{\mathrm{deg}}

\newcommand{\Id}{\rm{Id}}

\newcommand{\br}[1]{   [ \cdot,    \cdot  ]   }

\newcommand{\dt}{\delta^{T}}

\newcommand{\Hom}{\mathrm{Hom}}

\newcommand{\Sym}{\mathrm{Sym}}
\newcommand{\Nat}{\mathbb N}

\newcommand{\gl}{\mathfrak {gl}}

\newcommand{\ad}{\mathrm{ad}}

\newcommand{\sgn}{\mathrm{sgn}}

\newcommand{\K}{\mathbb{K}}

\newcommand{\perm}{\mathbb S}
\newcommand{\nat}{\mathbb Z}

\newcommand{\GRB}{\huaO}
\newcommand{\GRBN}{\mathrm{\huaO N}}

\newcommand{\KVN}{\mathrm{KVN}}

\newcommand{\KVB}{\mathrm{KV\Omega}}

\newcommand{\HN}{\mathrm{H N}}

\newcommand{\MN}{\mathrm{MN}}

\begin{document}

\title[Pre-Lie analogues of Poisson-Nijenhuis structures and Maurer-Cartan equations]{Pre-Lie analogues of Poisson-Nijenhuis structures and Maurer-Cartan equations}

\author{Jiefeng Liu}
\address{School of Mathematics and Statistics, Northeast Normal University, Changchun 130024, China}
\email{liujf12@126.com}

\author{Qi Wang}
\address{Department of Mathematics, Jilin University, Changchun 130012, Jilin, China}
\email{wangqi17@mails.jlu.edu.cn}
\vspace{-5mm}


\begin{abstract}
In this paper, we study pre-Lie analogues of Poisson-Nijenhuis structures and introduce $\GRBN$-structures on bimodules over pre-Lie algebras. We show that an $\GRBN$-structure gives rise to a hierarchy of pairwise compatible $\huaO$-operators. We study solutions of the strong Maurer-Cartan equation on the twilled pre-Lie algebra associated to an $\huaO$-operator, which gives rise to a pair of $\GRBN$-structures which are naturally in duality. We show that $\KVN$-structures and $\HN$-structures on a pre-Lie algebra $\g$ are corresponding to $\GRBN$-structures on the bimodule $(\g^*;\ad^*,-R^*)$, and $\KVB$-structures are corresponding to solutions of the strong Maurer-Cartan equation on a twilled pre-Lie algebra associated to an $\frks$-matrix.

\end{abstract}


\keywords{pre-Lie algebra, deformation, Nijenhuis structure, $\GRBN$-structure, Maurer-Cartan equation}

\footnotetext{{\it{MSC}}: 17A30, 17A60, 17B38}
\maketitle

\tableofcontents

\allowdisplaybreaks


\section{Introduction}\label{sec:intr}

Pre-Lie algebras (or left-symmetric algebras) are a class of nonassociative algebras coming from the study of convex homogeneous cones, affine manifolds and affine structures on Lie groups, deformation of associative algebras and then  appeared in many fields in mathematics and mathematical physics, such as complex and symplectic structures on Lie groups and Lie algebras, integrable systems, Poisson brackets and infinite dimensional Lie algebras, vertex algebras, quantum field theory, homotopy algebras, operads and $F$-manifold algebras. See \cite{Bakalov,Ban,ChaLiv,Dot,DSV,Lichnerowicz,LBS20,MT} and the survey paper \cite{Pre-lie algebra in geometry} for more details.

A Rota-Baxter operator of weight zero on a Lie algebra was introduced in the 1980s as the operator form of the classical Yang-Baxter equation,
named after the physicists C. N. Yang and R. Baxter \cite{BaR,Ya}.  The classical Yang-Baxter equation plays important roles in many fields in mathematics and mathematical physics such as integrable systems and quantum groups \cite{CP,STS}. See~\cite{Gub} for more details on Rota-Baxter operators. To better understand the classical Yang-Baxter equation and
the related integrable systems, the more general notion of an $\huaO$-operator
on a Lie algebra was introduced by Kupershmidt~\cite{Ku}, which can be traced back to Bordemann
\cite{Bor}.  An $\huaO$-operator gives rise to a skew-symmetric $r$-matrix in a larger Lie algebra \cite{Bai-1}. In the associative algebra context, such a structure was introduced by Uchino in \cite{Uchino}, which plays an important role in the bialgebra theory \cite{Bai-2,BGN2013A} and  leads to the splitting of operads~\cite{Bai-Bellier-Guo-Ni}. The notion of an $\huaO$-operator on a pre-Lie algebra was introduced by Bai and coauthors in \cite{Bai-Liu-Ni} for the purpose of studying the bialgebra theory of pre-Lie algebras and $L$-dendriform algebras.
\begin{defi}
    A linear map $T:V\longrightarrow \g$ is called an {\bf $\GRB$-operator} on a bimodule
    $(V;\huaL,\huaR)$ over a pre-Lie algebra $(\g,\cdot_\g)$  if it satisfies
  \begin{equation}\label{O-operator}
    T(u)\cdot_\g T(v)=T(\huaL_{T(u)}v+\huaR_{T(v)}u),\quad\forall~ u,v\in V.
  \end{equation}
\end{defi}
When $V=\g$ as a regular bimodule, an $\GRB$-operator reduces to a Rota-Baxter operator on a pre-Lie algebra \cite{LHB}. An $\GRB$-operator over a pre-Lie algebra can also be viewed as an analogue of a Poisson structure. Let $\g\ltimes_{\huaL,\huaR} V$ be a semi-direct product pre-Lie algebra and denote this structure by $\mu$. Then $C^*(V,\g)$ with the graded bracket
$${[f_1,f_2]_{\hat{\mu}}}=(-1)^{m-1}[[\hat{\mu},\hat{f}_1]^{\MN},\hat{f}_2]^{\MN},\quad\forall~f_1\in C^m(V,\g),~f_2\in C^n(V,g)$$
is a graded Lie algebra, where $\hat{\mu}$ is the horizontal lift of $\mu$ and the bracket $[-,-]^\MN$ is the Matsushima-Nijenhuis bracket on $C^*(\g\oplus V,\g\oplus V)$. It was shown in \cite{Liu19} that $T$ is an $\GRB$-operator if and only if it is a solution of the Maurer-Cartan equation, $  \half[\hat{T},\hat{T}]_{\hat{\mu}}=0$, where $\hat{T}$ is the horizontal lift of $T$. This parallels to the fact that $\pi$ is a Poisson structure on a manifold $M$  if and only if  $\pi$ is a solution of a Maurer-Cartan equation, $\half[\pi, \pi]_{\rm SN} = 0$, where the graded Lie bracket is the Schouten-Nijenhuis bracket of multi-vector fields.

Poisson-Nijenhuis structures were defined by Magri and Morosi in 1984 in their study of completely integrable systems \cite{MM}. The importance of Poisson-Nijenhuis structures is that they produce bi-Hamiltonian systems. See \cite{Kosmann1,Kosmann2} for more details on Poisson-Nijenhuis structures. The first purpose of this paper is to study the theory of analogues of Possion-Nijenhuis structures on pre-Lie algebras. We introduce the notion of an $\GRBN$-structure on a bimodule over a pre-Lie algebra, which consists of an $\huaO$-operator and a Nijenhuis structure on a bimodule over a pre-Lie algebra satisfying some compatibility conditions. The notion of a Nijenhuis structure on a $\g$-bimodule $V$ consists of a Nijenhuis operator $N$ on $\g$ and a linear map $S\in\gl(V)$ satisfying some compatibility conditions.  We point out that the introduction of the linear map ``$S$''  provides a general framework for the study of pre-Lie analogues of Poisson-Nijenhuis structures.

Another purpose of this paper is to generalize a theorem of Vaisman to pre-Lie algebras, in which he used solutions of the
strong Maurer-Cartan equation  to construct Poisson-Nijenhuis structures \cite{Vaisman}. Roughly speaking, for a Poisson manifold $(M,\pi)$, there is an induced Lie
algebroid structure on the cotangent bundle $T^*M$ and thus has a certain graded Lie bracket $[-,-]_\pi$ on the space of sections
of $\wedge^*T^*M$. Then Vaisman showed that if a $2$-form $\omega$ is a solution of the strong Maurer-Cartan equation, $d\omega=[\omega,\omega]_\pi=0$, then $N=\pi\circ \omega:TM\rightarrow TM$ is a Nijenhuis operator and $(\pi,N)$ is a Poisson-Nijenhuis structure. Similarly, for an $\huaO$-operator $T$ on a $\g$-bimodule $V$, there is an induced pre-Lie algebra structure on $V$ and denote this pre-Lie algebra by $V_T$. One can show that $\g\oplus V_T$ has a twilled pre-Lie algebra structure (see Theorem \ref{thm:Uchino}). Thus a differential graded Lie algebra structure, $(d_{\hat{\mu}_1},[-,-]_{{\hat{\mu}_2}})$, is induced on $C^*(\g,V_T)$ (see Theorem \ref{quasi-as-shLie}). By analogy with Vaisman's theorem, we assume
that $\Omega:\g\rightarrow V$ is a solution of the strong Maurer-Cartan equation, $d_{\hat{\mu}_1}\hat{ \Omega}=\half{[\hat{\Omega},\hat{\Omega}]_{\hat{\mu}_2}}=0$, where $\hat{\Omega}$ is the horizontal lift of $\Omega$, then we can show that $(N=T\circ\Omega,S=\Omega\circ T)$ is a Nijenhuis structure and $(T,N,S)$ is an $\GRBN$-structure (see Theorem \ref{thm:MC-GRBN}). This theorem can be considered as a pre-Lie version of Vaisman's result.

In Section \ref{sec:Pre}, first we recall some basic properties of pre-Lie algebras like bimodules, cohomology theory, Nijenhuis operators and Matsushima-Nijenhuis brackets for pre-Lie algebras. Then we recall the notions of Maurer-Cartan equation and the strong Maurer-Cartan equation on a twilled pre-Lie algebra. See \cite{CM,Uchino} and \cite{Kosmann88,KosmannD} for more details on twilled associative algebras and Lie algebras respectively.

In Section \ref{sec:deformation}, to define an $\GRBN$-structure on a bimodule $(V;\huaL,\huaR)$ over a pre-Lie algebra $\g$, first we study the cohomology of pre-Lie algebras with bimodules, then we study infinitesimal deformations of bimodules over pre-Lie algebras and introduce the notion of a Nijenhuis structure, which gives rise to a trivial deformation of the bimodule $(V^*;\huaL^*-\huaR^*,-\huaR^*)$ over  $\g$.  See
\cite{Ge0,Ge}, \cite{WBLS} and \cite{Sale} for more details {on}
deformations of associative algebras, pre-Lie algebras and Lie algebras respectively.

In Section \ref{sec:ON-structure}, we introduce some compatibility conditions between an $\GRB$-operator $T$ and a Nijenhuis structure $(N,S)$ to define the notion of an $\GRBN$-structure on a bimodule over a pre-Lie algebra. See \cite{HLS} for more details on $\GRBN$-structures on Lie algebras. Then we study the relations between $\KVN$-structures and $\HN$-structures on pre-Lie algebras and $\GRBN$-structures on the bimodule $(\g^*;\ad^*,-R^*)$. See \cite{WLS} for more details about $\KVN$-structures and $\HN$-structures on left-symmetric algebroids. At last, we study the relations between compatible $\GRB$-operators and $\GRBN$-structures, then various examples are given.

In Section \ref{sec:MC-ON}, we build the theory as a pre-Lie version of Vaisman's result. First we construct a twilled pre-Lie algebra $\g\bowtie V_T$ through an $\GRB$-operator $T:V\lon \g$. Then we show that a solution of the strong Maurer-Cartan equation on the twilled pre-Lie algebra $\g\bowtie V_T$ gives rise to an $\GRBN$-structure $(T, N=T\circ \Omega, S=\Omega\circ T)$. Furthermore, there is another $\GRBN$-structure $(\Omega, S, N)$ on a new pre-Lie algebra $V_T$. These two structures are dual to each other in some sense. At last, we introduce the notion of a $\KVB$-structure, which consists of an $\frks$-matrix and a symmetric $2$-cocycle satisfying some compatibility conditions and then show that $(r,\frkB)$ is a $\KVB$-structure if and only if $\frkB^\natural:\g\rightarrow \g^*$ is a solution of the strong Maurer-Cartan equation on the twilled pre-Lie algebra $\g\bowtie\g^*_{r^\sharp}$. The relations of $\KVN$-structures, $\HN$-structures and  $\KVB$-structures are discussed.

In this paper, all the vector spaces are over algebraically closed field $\mathbb K$ of characteristic $0$, and finite dimensional.
\vspace{2mm}

\noindent

{\bf Acknowledgements. }This research was  supported by NSFC (11901501).
\section{Preliminaries}\label{sec:Pre}
\subsection{Cohomology of pre-Lie algebras and Matsushima-Nijenhuis brackets}
\begin{defi}  A {\bf pre-Lie algebra} is a pair $(\g,\cdot_\g)$, where $\g$ is a vector space and  $\cdot_\g:\g\otimes \g\longrightarrow \g$ is a bilinear multiplication
satisfying that for all $x,y,z\in \g$, the associator
$(x,y,z)=(x\cdot_\g y)\cdot_\g z-x\cdot_\g(y\cdot_\g z)$ is symmetric in $x,y$,
i.e.
$$(x,y,z)=(y,x,z),\;\;{\rm or}\;\;{\rm
equivalently,}\;\;(x\cdot_\g y)\cdot_\g z-x\cdot_\g(y\cdot_\g z)=(y\cdot_\g x)\cdot_\g
z-y\cdot_\g(x\cdot_\g z).$$
\end{defi}

Let $(\g,\cdot_\g)$ be a pre-Lie algebra. The commutator $
[x,y]_\g=x\cdot_\g y-y\cdot_\g x$ defines a Lie algebra structure
on $\g$, which is called the {\bf sub-adjacent Lie algebra} of
$(\g,\cdot_\g)$ and denoted by $\g^c$. Furthermore,
$L:\g\longrightarrow \gl(\g)$ with $x\mapsto L_x$, where
$L_xy=x\cdot_\g y$, for all $x,y\in \g$, gives a module of
the Lie algebra $\g^c$ of $\g$. See \cite{Pre-lie algebra in
geometry} for more details.

\begin{defi}{\rm(\cite{WBLS})}
A {\bf Nijenhuis operator} on a
pre-Lie algebra $(\g,\cdot_\g)$ is a linear map $N:\g\longrightarrow \g$ satisfying
\begin{equation}
  N(x)\cdot_\g N(y)=N\big(N(x)\cdot_\g y+x\cdot_\g N(y)-N(x\cdot_\g y)\big),\quad \forall x,y\in \g.
\end{equation}
\end{defi}
The deformed operation $\cdot_N:\otimes^2\g\longrightarrow \g$ given by
\begin{equation}\label{eq:deformbracket}
  x\cdot_N y=N(x)\cdot_\g y+x\cdot_\g N(y)-N(x\cdot_\g y)
\end{equation}
 is a pre-Lie algebra multiplication and $N$ is a pre-Lie algebra homomorphism from $(\g,\cdot_N)$ to $(\g,\cdot_\g)$.

By direct calculations, we have
\begin{lem}\label{lem:Niejproperty}
  Let $(\g,\cdot_\g)$ be a pre-Lie algebra and $N$ a Nijenhuis operator on $\g$. For all $k,l\in\Nat$,
  \begin{itemize}
\item[$\rm(i)$]$(\g,\cdot_{N^k})$ is a pre-Lie algebra;
\item[$\rm(ii)$]$N^l$ is also a Nijenhuis operator on the pre-Lie algebra $(\g,\cdot_{N^k})$;
\item[$\rm(iii)$]The pre-Lie algebras $(\g,(\cdot_{N^k})_{N^l})$ and $(A,\cdot_{N^{k+l}})$ coincide;
\item[$\rm(iv)$]The pre-Lie algebras $(\g,\cdot_{N^k})$ and $(\g,\cdot_{N^l})$ are
compatible, that is,
any linear combination of $\cdot_{N^k}$ and $\cdot_{N^l}$ still
makes $\g$ into a pre-Lie algebra;
\item[$\rm(v)$]$N^l$ is a pre-Lie algebra homomorphism from $(\g,\cdot_{N^{k+l}})$ to $(\g,\cdot_{N^k})$.
  \end{itemize}
\end{lem}

\begin{defi}
  Let $(\g,\cdot_\g)$ be a pre-Lie algebra and $V$   a vector space. Let $\huaL,\huaR:\g\longrightarrow\gl(V)$ be two linear maps with $x\mapsto \huaL_x$ and $x\mapsto \huaR_x$ respectively. The triple $(V;\huaL,\huaR)$ is called a {\bf bimodule} over $\g$ if
\begin{eqnarray}
\label{representation condition 1} \huaL_x\huaL_yu-\huaL_{x\cdot_\g y}u&=&\huaL_y\huaL_xu-\huaL_{y\cdot_\g x}u,\\
 \label{representation condition 2}\huaL_x\huaR_yu-\huaR_y\huaL_xu&=&\huaR_{x\cdot_\g y}u-\huaR_y\huaR_xu, \quad \forall~x,y\in \g,~ u\in V.
\end{eqnarray}
 \end{defi}

In the sequel, we will simply call $V$ a $\g$-bimodule if there is no confusions possible.

In fact, $(V;\huaL,\huaR)$ is a bimodule of a pre-Lie algebra
$\g$ if and only if the direct sum $\g\oplus V$ of vector spaces is
 a pre-Lie algebra (the semi-direct product) by
defining the multiplication on $\g\oplus V$ by
$$
  (x_1+v_1)\cdot_{(\huaL,\huaR)}(x_2+v_2)=x_1\cdot_\g x_2+\huaL_{x_1}v_2+\huaR_{x_2}v_1,\quad \forall~ x_1,x_2\in \g,v_1,v_2\in V.
$$
  We denote it by $\g\ltimes_{\huaL,\huaR} V$ or simply by $\g\ltimes V$.

By a straightforward calculation, we have
\begin{pro}\label{pro:dual-module equiv}
 $(V;\huaL,\huaR)$ is a bimoudle over the pre-Lie algebra $(\g,\cdot_\g)$ if and only if $(V^*;\huaL^*-\huaR^*,-\huaR^*)$ is a bimoudle over the pre-Lie algebra $\g$,  where $\huaL^*:\g\longrightarrow \gl(V^*)$ and $\huaR^*:\g\longrightarrow \gl(V^*)$ are given by
$$
 \langle \huaL^*_x\xi,u\rangle=-\langle \xi,\huaL_x u\rangle,\quad\langle \huaR^*_x\xi,u\rangle=-\langle \xi,\huaR_xu\rangle,\quad \forall~ x\in \g,\xi\in V^*,u\in V.
$$
\end{pro}
In the following, $(V^*;\huaL^*-\huaR^*,-\huaR^*)$ is called the {\bf dual bimodule} of the bimodule $(V;\huaL,\huaR)$.

It is obvious that $(  \K  ;\rho=0,\mu=0)$ is a bimodule, which we
call the {\bf trivial bimodule}. Let $R:\g\rightarrow
\gl(\g)$ be a linear map with $x\mapsto R_x$, where the
linear map $R_x:\g\longrightarrow\g$  is defined by
$R_x(y)=y\cdot_\g x,$ for all $x, y\in \g$. Then
$(\g;\huaL=L,\huaR=R)$ is a bimodule, which we call the
{\bf regular bimodule}. The dual bimodule of the regular bimodule $(\g;L,R)$ is just the bimodule $(\g^*;{\rm ad}^*=L^*-R^*, -R^*)$ over $\g$.

The cohomology complex for a pre-Lie algebra $(\g,\cdot_\g)$ with a bimodule $(V;\huaL,\huaR)$ is given as follows.
The set of $n$-cochains is given by
$\Hom(\wedge^{n-1}\g\otimes \g,V),\
n\geq 1.$  For all $\phi\in \Hom(\wedge^{n-1}\g\otimes \g,V)$, the coboundary operator $\delta:\Hom(\wedge^{n-1}\g\otimes \g,V)\longrightarrow \Hom(\wedge^{n}\g\otimes \g,V)$ is given by
 \begin{eqnarray}\label{eq:pre-Lie cohomology}
 \nonumber\delta\phi(x_1, \cdots,x_{n+1})&=&\sum_{i=1}^{n}(-1)^{i+1}\huaL_{x_i}\phi(x_1, \cdots,\hat{x_i},\cdots,x_{n+1})\\
\label{eq:cobold} &&+\sum_{i=1}^{n}(-1)^{i+1}\huaR_{x_{n+1}}\phi(x_1, \cdots,\hat{x_i},\cdots,x_n,x_i)\\
 \nonumber&&-\sum_{i=1}^{n}(-1)^{i+1}\phi(x_1, \cdots,\hat{x_i},\cdots,x_n,x_i\cdot_\g x_{n+1})\\
 \nonumber&&+\sum_{1\leq i<j\leq n}(-1)^{i+j}\phi([x_i,x_j]_\g,x_1,\cdots,\hat{x_i},\cdots,\hat{x_j},\cdots,x_{n+1}),
\end{eqnarray}
for all $x_i\in \g,~i=1,\cdots,n+1$. We use the
symbol $\delta^T$ to refer the coboundary operator associated to
the trivial representation.

A permutation $\sigma\in\perm_n$ is called an $(i,n-i)$-unshuffle if $\sigma(1)<\cdots<\sigma(i)$ and $\sigma(i+1)<\cdots<\sigma(n)$. If $i=0$ and $i=n$, we assume $\sigma=\Id$. The set of all $(i,n-i)$-unshuffles will be denoted by $\perm_{(i,n-i)}$. The notion of an $(i_1,\cdots,i_k)$-unshuffle and the set $\perm_{(i_1,\cdots,i_k)}$ are defined similarly.

Let $\g$ be a vector space. We consider the graded vector space $C^*(\g,\g)=\oplus_{n\ge 1}C^n(\g,\g)=\oplus_{n\ge 1}\Hom(\wedge^{n-1}\g\otimes\g,\g)$. It was shown in \cite{ChaLiv,Nij,WBLS} that $C^*(\g,\g)$ equipped with the Matsushima-Nijenhuis bracket
\begin{eqnarray}\label{eq:Nij-preLie}
{[P,Q]^{\MN}}&=&P\diamond Q-(-1)^{pq}Q\diamond P,\quad\forall P\in C^{p+1}(\g,\g),Q\in C^{q+1}(\g,\g)
\end{eqnarray}
is a graded Lie algebra, where $P\diamond Q\in C^{p+q+1}(\g,\g)$ is defined by
\begin{eqnarray*}
&&P\diamond Q(x_1,\ldots,x_{p+q+1})\\
&=&\sum_{\sigma\in\perm_{(q,1,p-1)}}\sgn(\sigma)P(Q(x_{\sigma(1)},\ldots,x_{\sigma(q)},x_{\sigma(q+1)}),x_{\sigma(q+2)},\ldots,x_{\sigma(p+q)},x_{p+q+1})\\
&&+(-1)^{pq}\sum_{\sigma\in\perm_{(p,q)}}\sgn(\sigma)P(x_{\sigma(1)},\ldots,x_{\sigma(p)},Q(x_{\sigma(p+1)},\ldots,x_{\sigma(p+q)},x_{p+q+1})).
\end{eqnarray*}
In particular, $\pi\in\Hom(\otimes^2\g,\g)$ defines a pre-Lie algebra if and only if $[\pi,\pi]^{\MN}=0.$ If $\pi$ is a pre-Lie algebra structure, then $d_{\pi}(f):=[\pi,f]^{\MN}$ is a graded derivation of the graded Lie algebra $(C^*(\g,\g),[-,-]^{\MN})$ satisfying $d_{\pi}\circ d_{\pi}=0$, so that $(C^*(\g,\g),[-,-]^{\MN},d_{\pi})$ becomes a differential graded Lie algebra.

\subsection{Maurer-Cartan equations on twilled pre-Lie algebras}
 Let $\g_1$ and $\g_2$ be vector spaces and elements in $\g_1$ will be denoted by $x,y, x_i$ and elements in $\g_2$ will be denoted by $u,v,v_i$. Let $c:\wedge^{n-1}\g_1\otimes\g_1\lon \g_2$ be a linear map. We can construct a linear map $\hat{c}\in C^n(\g_1\oplus\g_2,\g_1\oplus\g_2)$ by
\begin{eqnarray*}
\hat{c}\big(x_1+v_1,\cdots,x_n+v_n\big):=c(x_1,\cdots,x_n).
\end{eqnarray*}
In general, for a given linear map $f:\wedge^{k-1}\g_1\otimes\wedge^{l}\g_2\otimes \g_1\lon\g_j$ for $j\in \{1,2\}$, we define a linear map $\hat{f}\in C^{k+l}(\g_1\oplus\g_2,\g_1\oplus\g_2)$ by
\begin{equation*}
 \hat{f}(x_1+v_1,x_2+v_2,\cdots,x_{k}+v_{k})=\sum_{\sigma\in\perm_{(k-1,l)}}\sgn(\tau)f(x_{\tau(1)},\cdots,x_{\tau(k-1)},v_{\tau(k)},\cdots,v_{\tau(k+l-1)},x_{k+l}).
\end{equation*}
Similarly, for $f:\wedge^{k}\g_1\otimes\wedge^{l-1}\g_2\otimes \g_2\lon\g_j$ for $j\in \{1,2\}$, we define a linear map $\hat{f}\in C^{k+l}(\g_1\oplus\g_2,\g_1\oplus\g_2)$ by
\begin{equation*}
 \hat{f}(x_1+v_1,x_2+v_2,\cdots,x_{k}+v_{k})=\sum_{\sigma\in\perm_{(k,l-1)}}\sgn(\tau)f(x_{\tau(1)},\cdots,x_{\tau(k)},v_{\tau(k+1)},\cdots,v_{\tau(k+l-1)},v_{k+l}).
\end{equation*}
We call the linear map $\hat{f}$ a {\bf horizontal lift} of $f$, or simply a lift. For example, the lifts of linear maps $\alpha:\g_1\otimes\g_1\lon\g_1,~\beta:\g_1\otimes\g_2\lon\g_2$ and $\gamma:\g_2\otimes\g_1\lon\g_2$ are defined respectively by
\begin{eqnarray}
\label{semidirect-1}\hat{\alpha}\big(x_1+v_1,x_2+v_2\big)&=&\alpha(x_1,x_2),\\
\label{semidirect-2}\hat{\beta}\big(x_1+v_1,x_2+v_2\big)&=&\beta(x_1,v_2),\\
\label{semidirect-3}\hat{\gamma}\big(x_1+v_1,x_2+v_2\big)&=&\gamma(v_1,x_2).
\end{eqnarray}
We define $\huaG^{k,l}=\wedge^{k-1}\g_1\otimes\wedge^{l}\g_2\otimes \g_1+\wedge^{k}\g_1\otimes\wedge^{l-1}\g_2\otimes \g_2$. The vector space $\wedge^{n-1}{(\g_1\oplus\g_2)}\otimes (\g_1\oplus\g_2)$ is isomorphic to the direct sum of
$\huaG^{k,l}$, $k + l = n$.

\begin{defi}
A linear map $f\in \Hom\big(\wedge^{n-1}(\g_1\oplus\g_2)\otimes(\g_1\oplus\g_2),(\g_1\oplus\g_2)\big)$ has a {\bf bidegree} $k|l$, if the following four conditions hold:
\begin{itemize}
\item[\rm(i)] $k+l+1=n;$
\item[\rm(ii)] If $X$ is an element in $\huaG^{k+1,l}$, then $f(X)\in\g_1;$
\item[\rm(iii)] If $X$ is an element in $\huaG^{k,l+1}$, then $f(X)\in\g_2;$
\item[\rm(iv)] All the other case, $f(X)=0.$
\end{itemize}
We denote a linear map $f$ with bidegree $k|l$ by $||f||=k|l$.
\end{defi}

The linear maps $\hat{\alpha},~\hat{\beta},~\hat{\gamma}\in C^2(\g_1\oplus\g_2,\g_1\oplus\g_2)$ given by \eqref{semidirect-1}, \eqref{semidirect-2} and \eqref{semidirect-3} have the bidegree  $||\hat{\alpha}||=||\hat{\beta}||=||\hat{\gamma}||=1|0$. In our later study, the subspaces $C^{k|0}(\g_1\oplus\g_2,\g_1\oplus\g_2)$ will
be frequently used. By the above lift map, we have the following isomorphisms
$$C^{k|0}(\g_1\oplus\g_2,\g_1\oplus\g_2)\cong \Hom\big(\wedge^{k}\g_1\otimes\g_1,\g_1\big)\oplus\Hom\big(\wedge^{k-1}\g_1\otimes\g_2\otimes\g_1,\g_2\big)\oplus \Hom\big(\wedge^{k}\g_1\otimes\g_2,\g_2\big).$$
In particular, we have
\begin{eqnarray*}
C^{0|0}(\g_1\oplus\g_2,\g_1\oplus\g_2)&\cong& \Hom\big(\g_1,\g_1\big)\oplus \Hom\big(\g_2,\g_2\big);\\
C^{1|0}(\g_1\oplus\g_2,\g_1\oplus\g_2)&\cong& \Hom\big(\g_1\otimes\g_1,\g_1\big)\oplus \Hom\big(\g_2\otimes\g_1,\g_2\big)\oplus \Hom\big(\g_1\otimes\g_2,\g_2\big).
\end{eqnarray*}

The following lemmas shows that the Matsushima-Nijenhuis bracket is compatible with the
bidegree.
\begin{lem}\label{important-lemma-2}
If $||f||=k_f|l_f$ and $||g||=k_g|l_g$, then $[f,g]^{\MN}$ has the bidegree $k_f+k_g|l_f+l_g.$
\end{lem}

Next we recall the definition of a twilled pre-Lie algebra.
\begin{defi}
  Let $(\huaG,\diamond)$ be a pre-Lie algebra that admits a decomposition into two subspaces, $\huaG=\g_1\oplus \g_2$. The triple $(\huaG,\g_1,\g_2)$ is called a {\bf twilled pre-Lie algebra}, if $\g_1$ and $\g_2$ are subalgebras of $\huaG$. We sometimes denote a twilled pre-Lie algebra by $\g_1\bowtie \g_2$.
\end{defi}

It is not hard to see that if $\g_1\bowtie \g_2$ is a twilled pre-Lie algebra, then there are linear maps $\huaL^1,\huaR^1:\g_1\rightarrow \gl(\g_2)$ and $\huaL^2,\huaR^2:\g_2\rightarrow \gl(\g_1)$ such that $(\g_2;\huaL^1,\huaR^1)$ is a bimodule over $(\g_1,\diamond_1)$  and $(\g_1;\huaL^2,\huaR^2)$ is a bimodule over $(\g_2,\diamond_2)$, where $\diamond_1$ and $\diamond_2$ are the multiplication $\diamond$ of $\huaG$ restricted to $\g_1$ and $\g_2$, respectively. Their bimodule structures are defined by the following decomposition of the pre-Lie multiplication of $\huaG$:
$$
  x\diamond u=\huaL^1_xu+\huaR^2_ux,\quad u\diamond x=\huaL^2_ux+\huaR^1_x u,\quad\forall~x\in \g_1,u\in \g_2.
$$

Denote by $\mu_1$ the semidirect product pre-Lie algebra structure on $\g_1\ltimes_{\huaL^1,\huaR^1} \g_2$ and $\mu_2$ the semidirect product pre-Lie algebra structure on $\g_2\ltimes_{\huaL^2,\huaR^2} \g_1$. Note that the graded vector space $C^*(\huaG,\huaG)=\oplus_{n\ge 1}C^n(\huaG,\huaG)=\oplus_{n\ge 1}\Hom(\wedge^{n-1}\huaG\otimes\huaG,\huaG)$  equipped with the Matsushima-Nijenhuis bracket is a graded Lie algebra. Set $C^*(\g_1,\g_2)=\oplus_{n\ge 1}C^n(\g_1,\g_2)=\oplus_{n\ge 1}\Hom(\wedge^{n-1}\g_1\otimes\g_1,\g_2)$.

\begin{thm}{\rm(\cite{Liu19})}\label{quasi-as-shLie}
Let $(\huaG,\g_1,\g_2)$ be a twilled pre-Lie algebra. We define $d_{\hat{\mu}_1}:C^m(\g_1,\g_2)\lon C^{m+1}(\g_1,\g_2)$ and  $[-,-]_{{\hat{\mu}_2}}:C^m(\g_1,\g_2)\times C^n(\g_1,\g_2)\lon C^{m+n}(\g_1,\g_2)$ by
\begin{eqnarray}
\label{eq:shLie1}d_{\hat{\mu}_1}(f_1)&=&[\hat{\mu}_1,\hat{f}_1]^{\MN},\\
\label{eq:shLie2}{[f_1,f_2]_{\hat{\mu}_2}}&=&(-1)^{m-1}[[\hat{\mu}_2,\hat{f}_1]^{\MN},\hat{f}_2]^{\MN}
\end{eqnarray}
for all $f_1\in C^m(\g_1,\g_2),~f_2\in C^n(\g_1,\g_2).$ Then $(C^*(\g_1,\g_2),d_{\hat{\mu}_1},[-,-]_{{\hat{\mu}_2}})$ is a differential graded Lie algebra.
\end{thm}

\begin{defi}
 Let $(\huaG,\g_1,\g_2)$ be a twilled pre-Lie algebra and $\Omega:\g_1\longrightarrow \g_2$ a linear map. The equations
  \begin{equation*}
  d_{\hat{\mu}_1} \hat{\Omega}+\half{[\hat{\Omega},\hat{\Omega}]_{\hat{\mu}_2}}=0,\quad d_{\hat{\mu}_1} \hat{\Omega}=\half{[\hat{\Omega},\hat{\Omega}]_{\hat{\mu}_2}}=0
  \end{equation*}
  are called  the {\bf Maurer-Cartan equation}  and
  the   {\bf strong Maurer-Cartan equation} respectively.
\end{defi}

\begin{lem}\label{lem:SCM-condition}
  Let $(\huaG,\g_1,\g_2)$ be a twilled pre-Lie algebra and $\Omega:\g_1\longrightarrow \g_2$ a linear map. Then $\Omega$ is a solution of the Maurer-Cartan equation
 if and only if $\Omega$ satisfies
\begin{equation}\label{eq:tMC-1}
    \Omega(x)\diamond_2\Omega(y)+\huaL^1_{x}\Omega(y)+\huaR^1_{y}\Omega(x)=\Omega\big(\huaL^2_{  \Omega(x)}y+\huaR^2_{  \Omega(y)}x\big)+\Omega(x\diamond_1 y),\quad\forall~x,y\in \g_1.
\end{equation}
 $\Omega$ is a solution of the strong Maurer-Cartan equation
  if and only if $\Omega$ satisfies
\begin{eqnarray}
\label{eq:tMC-3}   \Omega(x\diamond_1 y)&=&\huaL^1_{x}\Omega(y)+\huaR^1_{y}\Omega(x),\\
\label{eq:tMC-2}\Omega(x)\diamond_2\Omega(y)&=&\Omega\big(\huaL^2_{  \Omega(x)}y+\huaR^2_{  \Omega(y)}x\big).
\end{eqnarray}
\end{lem}
\begin{proof}
  By direct calculations, we have
  \begin{eqnarray*}
    d_{\hat{\mu}_1} \hat{\Omega}(x+u,y+v)&=&[\hat{\pi_1},\hat{\Omega}]^\MN(x+u,y+v)=\huaL^1_{x}\Omega(y)+\huaR^1_{y}\Omega(x)-\Omega(x\diamond_1 y),\\
   \half{[\hat{\Omega},\hat{\Omega}]_{\hat{\mu}_2}}(x+u,y+v)&=&[[\hat{\pi_2},\hat{\Omega}]^{\MN},\hat{\Omega}]^{\MN}(x+u,y+v)\\
    &=&\Omega(x)\diamond_2\Omega(y)-\Omega\big(\huaL^2_{  \Omega(x)}y+\huaR^2_{  \Omega(y)}x\big).
  \end{eqnarray*}
  Then we have
  \begin{eqnarray*}
    (d_{\hat{\mu}_1}\hat{\Omega}+\half{[\hat{\Omega},\hat{\Omega}]_{\hat{\mu}_2}})(x+u,y+v)&=&\huaL^1_{x}\Omega(y)+\huaR^1_{y}\Omega(x)-\Omega(x\diamond_1 y)+\Omega(x)\diamond_2\Omega(y)\\
    &&-\Omega\big(\huaL^2_{  \Omega(x)}y+\huaR^2_{  \Omega(y)}x\big).
  \end{eqnarray*}
Thus $\Omega$ is a solution of the Maurer-Cartan equation if and only if $\Omega$ satisfies \eqref{eq:tMC-1}. The second conclusion follows immediately.
\end{proof}

\section{Infinitesimal deformations of bimodules over pre-Lie algebras}\label{sec:deformation}
In this section, we study infinitesimal deformations of a bimodule over a pre-Lie algebra and introduce the notion of a Nijenhuis structure. First we study the cohomology of pre-Lie algebras with bimodules, which can describe the equivalence classes between infinitesimal deformations of a pre-Lie algebra with a bimodule.

\subsection{Cohomology of pre-Lie algebras with bimodules}
Let $\g$ and $V$ be two vector spaces and set $\huaG=\g\oplus V$.  Denote by $C^p(\huaG,\huaG)=\Hom(\wedge^{p-1}\huaG\otimes\huaG ,\huaG)$ and set $C^*(\huaG,\huaG)=\oplus_{p\in\Nat}C^p(\huaG,\huaG)$. The graded vector space $C^*(\huaG,\huaG)$ equipped with the Matsushima-Nijenhuis bracket $[-,-]^{\MN}$ is a graded Lie algebra. Furthermore, we have
\begin{lem}
 Let $\g$ and $V$ be two vector spaces. Then $(\oplus_{k=0}^{+\infty}C^{k|0}(\huaG,\huaG),[-,-]^{\MN})$ is a graded Lie subalgebra of $(C^*(\huaG,\huaG),[-,-]^{\MN})$.
\end{lem}
\begin{proof}
  By Lemma \ref{important-lemma-2}, for $||f||=k|0$ and $||g||=l|0$, $[f,g]^{\MN}$ has the bidegree $k+l|0$. Thus the Matsushima-Nijenhuis bracket restricted to $\oplus_{k=0}^{+\infty}C^{k|0}(\huaG,\huaG)$ is closed. The conclusion follows immediately.
\end{proof}

\begin{pro}\label{pro:bimodule and C-bracket}
Let $\g$ and $V$ be two vector spaces. Let $\pi:\g\otimes\g\rightarrow \g$ and $\huaL,\huaR:\g\longrightarrow\gl(V)$ be linear maps. Then $\hat{\pi}+\hat{\huaL}+\hat{\huaR}\in C^{1|0}(\huaG,\huaG)$ is a solution of the Maurer-Cartan equation of the  graded Lie algebra $(\oplus_{k=0}^{+\infty}C^{k|0}(\huaG,\huaG),[-,-]^{\MN})$, i.e.
\begin{equation}\label{eq:MN-semidirect product}
[\hat{\pi}+\hat{\huaL}+\hat{\huaR},\hat{\pi}+\hat{\huaL}+\hat{\huaR}]^{\MN}=0
\end{equation}
if and only if
$(\g,\pi)$ is a pre-Lie algebra and $(V;\huaL,\huaR)$ is a bimodule of the pre-Lie algebra $(\g,\pi)$.
\end{pro}
\begin{proof}
By direct calculations, we have
\begin{eqnarray*}
 &&\half[\hat{\pi}+\hat{\huaL}+\hat{\huaR},\hat{\pi}+\hat{\huaL}+\hat{\huaR}]^{\MN}(x_1+u_1,x_2+u_2,x_3+u_3)\\
 &=&\pi(\pi(x_1,x_2),x_3)-\pi(\pi(x_2,x_1),x_3)-\pi(x_1,\pi(x_2,x_3))+\pi(x_2,\pi(x_1,x_3))\\
 &&+\huaL_{\pi(x_1, x_2)}u_3-\huaL_{\pi(x_2, x_1)}u_3-\huaL_{x_1}\huaL_{x_2}u_3-\huaL_{x_2}\huaL_{x_1}u_3-\huaL_{x_1}\huaR_{x_3}u_2+\huaR_{x_3}\huaL_{x_1}u_2\\
 &&+\huaR_{x_1\cdot_\g x_3}u_2-\huaR_{x_3}\huaR_{x_1}u_2+\huaL_{x_2}\huaR_{x_3}u_1-\huaR_{x_3}\huaL_{x_2}u_1-\huaR_{x_2\cdot_\g x_3}u_1+\huaR_{x_3}\huaR_{x_2}u_1.
\end{eqnarray*}
Thus if $(\g,\pi)$ is a pre-Lie algebra and $(V;\huaL,\huaR)$ is a bimodule of the pre-Lie algebra $(\g,\pi)$, then  \eqref{eq:MN-semidirect product} holds.

Conversely, we have
\begin{eqnarray*}
 &&\half[\hat{\pi}+\hat{\huaL}+\hat{\huaR},\hat{\pi}+\hat{\huaL}+\hat{\huaR}]^{\MN}(x_1,x_2,x_3)\\
 &=&\pi(\pi(x_1,x_2),x_3)-\pi(\pi(x_2,x_1),x_3)-\pi(x_1,\pi(x_2,x_3))+\pi(x_2,\pi(x_1,x_3));\\
 &&\half[\hat{\pi}+\hat{\huaL}+\hat{\huaR},\hat{\pi}+\hat{\huaL}+\hat{\huaR}]^{\MN}(x_1,x_2,u_3)\\
 &=&\huaL_{\pi(x_1, x_2)}u_3-\huaL_{\pi(x_2, x_1)}u_3-\huaL_{x_1}\huaL_{x_2}u_3-\huaL_{x_2}\huaL_{x_1}u_3;\\
 &&\half[\hat{\pi}+\hat{\huaL}+\hat{\huaR},\hat{\pi}+\hat{\huaL}+\hat{\huaR}]^{\MN}(x_1,u_2,x_3)\\
 &=&-\huaL_{x_1}\huaR_{x_3}u_2+\huaR_{x_3}\huaL_{x_1}u_2+\huaR_{x_1\cdot_\g x_3}u_2-\huaR_{x_3}\huaR_{x_1}u_2.
\end{eqnarray*}
Thus if \eqref{eq:MN-semidirect product} holds, then $(\g,\pi)$ is a pre-Lie algebra and $(V;\huaL,\huaR)$ is a bimodule of the pre-Lie algebra $(\g,\pi)$.
\end{proof}

Let $(V;\huaL,\huaR)$ be a bimodule of the pre-Lie algebra $(\g,\pi)$ and set $\mu={\pi}+{\huaL}+{\huaR}$. For $n\geq1$, we define the set of $n$-cochains $\huaC^n(\g,\huaL,\huaR)=C^{n-1|0}(\g\oplus V,\g\oplus V)$, i.e.
$$\huaC^n(\g,\huaL,\huaR)= \Hom\big(\wedge^{n-1}\g\otimes\g,\g\big)\oplus\Hom\big(\wedge^{n-2}\g\otimes V\otimes\g,V\big)\oplus \Hom\big(\wedge^{n-1}\g\otimes V,V\big).$$
Define a coboundary operator $\partial:\huaC^n(\g,\huaL,\huaR)\rightarrow \huaC^{n+1}(\g,\huaL,\huaR)$ by
\begin{equation}\label{eq:cob}
\partial(\phi)=(-1)^{n-1}[\hat{\mu},\hat{\phi}]^{\MN},\quad\forall~\phi\in \huaC^n(\g,\huaL,\huaR).
\end{equation}
In fact, since $\mu\in C^{1|0}(\g\oplus V,\g\oplus V)$ and $\phi\in C^{n-1|0}(\g\oplus V,\g\oplus V)$, by Lemma \ref{important-lemma-2}, we have $\partial(\phi)\in\huaC^{n+1}(\g,\huaL,\huaR) $. By Proposition \ref{pro:bimodule and C-bracket}, we have $[\hat{\mu},\hat{\mu}]^{\MN}=0$. Because of the graded Jacobi identity, we have $\partial\circ \partial=0$. Thus we obtain a well-defined cochain complex $(\oplus_{n=1}^{+\infty}\huaC^n(\g,\huaL,\huaR),\partial)$.
 \begin{defi}
 The cohomology of the cochain complex $(\oplus_{n=1}^{+\infty}\huaC^n(\g,\huaL,\huaR),\partial)$ is called the cohomology of the pre-Lie algebra $\g$ with the bimodule $(V;\huaL,\huaR)$. The corresponding cohomology group is denoted by $\huaH^n(\g,\huaL,\huaR)$.
 \end{defi}

Now we give the precise formula for the coboundary operator $\partial$. For any $n$-cochain $\phi\in \huaC^n(\g,\huaL,\huaR)$, we will write $\phi=(\phi_1,\phi_2,\phi_3)$, where $\phi_1\in \Hom\big(\wedge^{n-1}\g\otimes\g,\g\big)$, $\phi_2\in\Hom\big(\wedge^{n-2}\g\otimes V\otimes\g,V\big)$
and $\phi_3\in \Hom\big(\wedge^{n-1}\g\otimes V,V\big)$. Then we have
$$\partial \phi=((\partial \phi)_1,(\partial  \phi)_2,(\partial  \phi)_3),$$
where $(\partial \phi)_1,(\partial  \phi)_2$ and $(\partial  \phi)_3$ are determined by
\begin{eqnarray*}
&&(\partial \phi)_1(x_1,\cdots,x_{n+1})=\delta \phi_1(x_1,\cdots,x_{n+1});\\
 &&(\partial\phi)_2(x_1,\cdots,x_{n-1},v,x_{n})=\sum_{i=1}^{n-1}(-1)^{i+1}\Big(\huaL_{x_i}\phi_2(x_1,\cdots,\hat{x_i},\cdots,v,x_{n})\\
 &&+\huaR_{x_n}\phi_2(x_1,\cdots,\hat{x_i},\cdots,v,x_{i})-\phi_2\big(x_1,\cdots,\hat{x_i},\cdots,v,x_i\cdot_{\g}x_{n}\big)\Big)\\
 &&(-1)^{n+1}\huaR_{\phi_1(x_1,\cdots,x_n)}v+(-1)^{n+1}\huaR_{x_n}\phi_3(x_1,\cdots,x_{n-1},v)+(-1)^{n}\phi_3(x_1,\cdots,x_{n-1},\huaR_{x_n}v)\\
  &&+\sum_{1\leq i<j\leq n-1}(-1)^{i+j}\phi_2\big([x_i,x_j]_\g,\cdots,\hat{x_i},\cdots,\hat{x_j},\cdots,v,x_n\big);\\
 &&(\partial\phi)_3(x_1,\cdots,x_{n},v)=\sum_{i=1}^{n}(-1)^{i+1}\big(\huaL_{\phi_1(x_1,\cdots,\hat{x_i},\cdots,x_n,x_{i})}v+\huaL_{x_i}\phi_3(x_1,\cdots,\hat{x_i},\cdots,x_{n},v)\big)\\
 &&-\sum_{i=1}^{n}(-1)^{i+1}\phi_3\big(x_1,\cdots,\hat{x_i},\cdots,x_n,\huaL_{x_i}v\big)+\sum_{1\leq i<j\leq n}(-1)^{i+j}\phi_3\big([x_i,x_j]_\g,\cdots,\hat{x_i},\cdots,\hat{x_j},\cdots,x_n,v\big),
\end{eqnarray*}
where $\delta$ is given by \eqref{eq:pre-Lie cohomology} and $x_i\in\g,v\in V$.

\emptycomment{By direct calculations, we have
\begin{pro}
For all $\phi\in C^{n}(\huaG,\huaG)$, we have
\begin{eqnarray*}\label{eq:pre-Lie cohomology2}
&&\nonumber\partial\phi(x_1+u_1,\cdots,x_{n+1}+u_{n+1})\\
\nonumber&=&\sum_{i=1}^{n}(-1)^{i+1}(x_i+u_i)\cdot_{\huaL,\huaR}\phi(x_1+u_1,\cdots,\hat{x_i}+\hat{u_i},\cdots,x_{n+1}+u_{n+1})\\
\nonumber&&+\sum_{i=1}^{n}(-1)^{i+1}\phi(x_1+u_1,\cdots,\hat{x_i}+\hat{u_i},\cdots,x_n+u_n,x_i+u_i)\cdot_{\huaL,\huaR}( x_{n+1}+u_{n+1})\\
\nonumber &&-\sum_{i=1}^{n}(-1)^{i+1}\phi\big(x_1+u_1,\cdots,\hat{x_i}+\hat{u_i},\cdots,x_n+u_n,(x_i+u_i)\cdot_{\huaL,\huaR}(x_{n+1}+u_{n+1})\big)\\
\nonumber &&+\sum_{1\leq i<j\leq n}(-1)^{i+j}\phi\big((x_i+u_i)\cdot_{\huaL,\huaR} (x_j+u_j)-(x_j+u_j)\cdot_{\huaL,\huaR} (x_i+u_i),\cdots,\hat{x_i}+\hat{u_i},\\
&&\cdots,\hat{x_j}+\hat{u_j},\cdots,x_{n+1}+u_{n+1}\big),
\end{eqnarray*}
where $x_i\in\g,u_i\in V$ for $i=1,2,\cdots,n+1$ and $\cdot_{(\huaL,\huaR)}$ is the semi-direct product pre-Lie algebra multiplication of $\g\ltimes_{\huaL,\huaR}V$.
\end{pro}}

\subsection{Infinitesimal deformations of bimodules over pre-Lie algebras}
Let $(V;\huaL,\huaR)$ be a bimodule over a pre-Lie algebra $(\g,\cdot_\g)$. Let $\omega:\otimes^2\g\longrightarrow \g$, $\sigma:\g\longrightarrow\gl(V)$ and $\tau:\g\longrightarrow\gl(V)$ be linear maps. Consider a $t$-parametrized family of multiplication operations and linear maps:
\begin{eqnarray*}
  x\cdot_t y=x\cdot_\g y+t\omega(x,y),\quad
  \huaL^t_x=\huaL_x+t\sigma(x),\quad
 \huaR^t_x=\huaR_x+t\tau(x), \quad \forall x,y\in \g.
\end{eqnarray*}
  If $(\g,\cdot_t)$ are pre-Lie algebras and $(V;\huaL^t,\huaR^t)$ are bimodules over $(\g,\cdot_t)$  for all $t\in \mathbb K$, we say that $(\omega,\sigma,\tau)$ generates an {\bf infinitesimal deformation} of the $\g$-bimodule $V$.

Let $\pi_t$ denote the pre-Lie algebra structure $(\g,\cdot_t)$. By Proposition \ref{pro:bimodule and C-bracket}, the bimodule $(V;\huaL^t,\huaR^t)$ over the pre-Lie algebra $(\g,\cdot_t)$ is an infinitesimal deformation of the $\g$-bimodule $V$ if and only if
\begin{eqnarray*}
 [\hat{\pi_t}+\hat{\huaL^t}+\hat{\huaR^t},\hat{\pi_t}+\hat{\huaL^t}+\hat{\huaR^t}]^{\MN}=0,
\end{eqnarray*}
which is equivalent to
\begin{eqnarray}
 \label{eq:A-bimodule1}{[\hat{\omega}+\hat{\sigma}+\hat{\tau},\hat{\omega}+\hat{\sigma}+\hat{\tau}]^{\MN}}&=&0,\\
 \label{eq:A-bimodule2}{[\hat{\pi}+\hat{\huaL}+\hat{\huaR},\hat{\omega}+\hat{\sigma}+\hat{\tau}]^{\MN}}&=&0.
\end{eqnarray}
By Proposition \ref{pro:bimodule and C-bracket}, \eqref{eq:A-bimodule1} means that $(\g,\omega)$ is a pre-Lie algebra and $(V;\sigma,\tau)$ is a bimodule of the pre-Lie algebra $(\g,\omega)$. \eqref{eq:A-bimodule2} means that $\hat{\omega}+\hat{\sigma}+\hat{\tau}$ is a $2$-cocycle for the pre-Lie algebra $\g$ with the bimodule $(V;\huaL,\huaR)$, i.e. $\partial(\hat{\omega}+\hat{\sigma}+\hat{\tau})=0$.

\begin{defi} Let $(V;\huaL^t,\huaR^t)$ and $(V;{\huaL'}^t,{\huaR'}^t)$ the bimodules over pre-Lie algebras $(\g,\cdot_t)$ and $(\g,\cdot'_t)$ respectively be two  infinitesimal deformations of a $\g$-bimodule $V$. We call them {\bf equivalent} if there exists $N\in\gl(\g)$ and $S\in\gl(V)$ such that $({\Id}_A+tN,{\Id}_V+tS)$ is  a  homomorphism from the bimodule $(V;{\huaL'}^t,{\huaR'}^t)$   to the bimodule $(V;\huaL^t,\huaR^t)$, i.e. for all $x,y\in \g$, there holds:
 \begin{eqnarray*}
   ({\Id}_A+tN)(x\cdot'_t y)&=&({\Id}_A+tN)(x)\cdot_t({\Id}_A+tN)(y),\\
   ({\Id}_V+tS)\circ {\huaL'}^t_x&=&\huaL^t_{({\Id}_A+tN)(x)}\circ ({\Id}_V+tS),\\
    ({\Id}_V+tS)\circ {\huaR'}^t_x&=&\huaR^t_{({\Id}_A+tN)(x)}\circ ({\Id}_V+tS).
 \end{eqnarray*}

 An  infinitesimal deformation of a $\g$-bimodule $V$ is said to be {\bf trivial} if it is equivalent to the $\g$-bimodule $V$.
 \end{defi}

By direct calculations, the bimodule $(V;\huaL^t,\huaR^t)$ over the pre-Lie algebra $(\g,\cdot_t)$ and the bimodule $(V;{\huaL'}^t,{\huaR'}^t)$ over the pre-Lie algebra $(\g,\cdot'_t)$ are equivalent deformations  if and only if
\begin{eqnarray*}
(\hat{\omega}+\hat{\sigma}+\hat{\tau})(x+u,y+v)-(\hat{\omega'}+\hat{\sigma'}+\hat{\tau'})(x+u,y+v)&=&\partial (\hat{N}+\hat{S})(x+u,y+v),\label{2-exac}\\
(\hat{\omega'}+\hat{\sigma'}+\hat{\tau'})(N(x)+S(u),N(y)+S(v))&=&0
\end{eqnarray*}
and
\begin{eqnarray*}
  &&(\hat{N}+\hat{S})(\hat{\omega}+\hat{\sigma}+\hat{\tau})(x+u,y+v)=(\hat{\omega'}+\hat{\sigma'}+\hat{\tau'})(x+u,N(y)+S(v))\\
&&\quad\quad\quad\quad+(\hat{\omega'}+\hat{\sigma'}+\hat{\tau'})(N(x)+S(u),y+v)+(\hat{\mu}+\hat{\huaL}+\hat{\huaR})(N(x)+S(u),N(y)+S(v)).
\end{eqnarray*}

We summarize the above discussion into the following
 conclusion:
\begin{thm}\label{thm:deformation}
Let $(V;\huaL^t,\huaR^t)$ the bimodule over the pre-Lie algebra $(\g,\cdot_t)$ be an infinitesimal deformation of a $\g$-bimodule $V$ generated by $(\omega,\sigma,\tau)$. Then $\hat{\omega}+\hat{\sigma}+\hat{\tau}\in \huaC^2(\g,\huaL,\huaR)$ is
closed, i.e. $\partial(\hat{\omega}+\hat{\sigma}+\hat{\tau})=0.$ Furthermore, if two infinitesimal deformations $(V;\huaL^t,\huaR^t)$ and $(V;{\huaL'}^t,{\huaR'}^t)$ over pre-Lie algebras $(\g,\cdot_t)$ and $(\g,\cdot'_t)$ generated by $(\omega,\sigma,\tau)$ and $(\omega',\sigma',\tau')$
 respectively  are equivalent,  then
$\hat{\omega}+\hat{\sigma}+\hat{\tau}$ and $\hat{\omega'}+\hat{\sigma'}+\hat{\tau'}$ are in the same cohomology class in $\huaH^2(\g,\huaL,\huaR)$.
\end{thm}

 One can deduce that the bimodule $(V;\huaL^t,\huaR^t)$ over the pre-Lie algebra $(\g,\cdot_t)$ is a trivial deformation if and only if for all $x,y\in \g,u,v\in V$, we have
\begin{eqnarray*}
(\hat{\omega}+\hat{\sigma}+\hat{\tau})(x+u,y+v)&=&\partial (\hat{N}+\hat{S})(x+u,y+v),\label{2-exact}\\
(\hat{N}+\hat{S})(\hat{\omega}+\hat{\sigma}+\hat{\tau})(x+u,y+v)&=&(\hat{\mu}+\hat{\huaL}+\hat{\huaR})(N(x)+S(u),N(y)+S(v)).
\end{eqnarray*}
Equivalently, we have
\begin{eqnarray}
\label{eq:trivialdeform1}\omega(x,y)&=&N(x)\cdot_\g y+x\cdot_\g N(y)-N(x\cdot_\g y),\\
\label{eq:trivialdeform2}N\omega(x,y)&=&N(x)\cdot_\g N(y),\\
\label{eq:trivialdeform3}\sigma(x)&=&\huaL_{N(x)}+\huaL_x\circ S-S\circ \huaL_x,\\
\label{eq:trivialdeform4}\huaL_{N(x)}\circ S&=&S\circ \sigma(x),\\
\label{eq:trivialdeform5}\tau(x)&=&\huaR_{N(x)}+\huaR_x\circ S-S\circ \huaR_x,\\
\label{eq:trivialdeform6}\huaR_{N(x)}\circ S&=&S\circ \tau(x).
\end{eqnarray}
It follows from \eqref{eq:trivialdeform1} and \eqref{eq:trivialdeform2} that $N$ must be a Nijenhuis operator on the pre-Lie algebra $(\g,\cdot_\g)$. It follows from \eqref{eq:trivialdeform3} and \eqref{eq:trivialdeform4} that $N$ and $S$ should satisfy the condition:
\begin{equation}\label{eq:Nijpair1}
     \huaL_{N(x)}S(v)=S(\huaL_{N(x)}v+\huaL_xS(v)-S(\huaL_xv)),\quad \forall x\in \g, v\in V.
   \end{equation}
It follows from \eqref{eq:trivialdeform5} and \eqref{eq:trivialdeform6} that $N$ and $S$ should also satisfy the condition:
\begin{equation}\label{eq:Nijpair2}
     \huaR_{N(x)}S(v)=S(\huaR_{N(x)}v+\huaR_xS(v)-S(\huaR_xv)),\quad \forall x\in \g, v\in V.
   \end{equation}

In a trivial infinitesimal deformation of a $\g$-bimodule $V$, $N$ is a Nijenhuis operator on the pre-Lie algebra $(\g,\cdot_\g)$ and conditions \eqref{eq:Nijpair1} and \eqref{eq:Nijpair2} hold. In fact, the converse is also true.
\begin{thm}\label{thm:trivial deform}
  Let $(V;\huaL,\huaR)$ be a bimodule over a pre-Lie algebra $(\g,\cdot_\g)$, $N\in\gl(\g)$ and $S\in\gl(V)$. If $N$ is a Nijenhuis operator on the pre-Lie algebra $(\g,\cdot_\g)$ \liu{and if $S$ satisfies \eqref{eq:Nijpair1} and \eqref{eq:Nijpair2}}, then a deformation of the $\g$-bimodule $V$ can be obtained by putting
  \begin{eqnarray*}
  \omega(x,y)&=&N(x)\cdot_\g y+x\cdot_\g N(y)-N(x\cdot_\g y),\\
\sigma(x)&=&\huaL_{N(x)}+\huaL_x\circ S-S\circ \huaL_x,\\
\tau(x)&=&\huaR_{N(x)}+\huaR_x\circ S-S\circ \huaR_x
\end{eqnarray*}
 for all $x, y\in \g$. Furthermore, this deformation is trivial.
\end{thm}
\begin{proof}
  It is a straightforward calculation. We omit the details.
\end{proof}

 It is not hard to check that
\begin{pro}
Let $(V;\huaL,\huaR)$ be a bimodule over a pre-Lie algebra $(\g,\cdot_\g)$. Then $N$ is a Nijenhuis operator on the pre-Lie algebra $(\g,\cdot_\g)$ \liu{and $S$ satisfies \eqref{eq:Nijpair1} and \eqref{eq:Nijpair2}} if and only if  $N+S$ is a Nijenhuis operator on
the semidirect product pre-Lie algebra $\g\ltimes_{\huaL,\huaR}
V.$
\end{pro}

 Now we shall give the main notion in this section.
 \begin{defi}
Let $(V;\huaL,\huaR)$ be a bimodule over a pre-Lie algebra
$(\g,\cdot_\g)$. A pair $(N,S)$, where $N\in\gl(\g)$ and $S\in\gl(V)$, is
called a {\bf Nijenhuis structure} on a $\g$-bimodule $V$ if $N$ and $S^*$ generate a trivial deformation of the bimodule $(V^*;\huaL^*-\huaR^*,-\huaR^*)$ over the pre-Lie algebra $(\g,\cdot_\g)$. Equivalently, $N$ is a Nijenhuis
operator on $(\g,\cdot_\g)$ and for all $x\in \g, v\in V$, the following conditions hold:
   \begin{eqnarray}
   \label{eq:coNijpair1}\huaL_{N(x)}S(v)&=&S(\huaL_{N(x)}v)+\huaL_{x}S^2(v)-S(\huaL_xS(v)),\\
   \label{eq:coNijpair2}\huaR_{N(x)}S(v)&=&S(\huaR_{N(x)}v)+\huaR_{x}S^2(v)-S(\huaR_xS(v)).
   \end{eqnarray}
 \end{defi}

 \begin{ex}\label{ex:cNP}
 {\rm Let $N$ be a Nijenhuis operator on a pre-Lie algebra $(\g,\cdot_\g)$. Then $(N,N^*)$ is a Nijenhuis structure on the bimodule $(\g^*;\ad^*,-R^*)$ over $(\g,\cdot_\g)$.}
 \end{ex}

 Given a Nijenhuis structure on a $\g$-bimodule $V$, we have the deformed pre-Lie algebra $(\g,\cdot_N)$. Next, we construct a bimodule of  $(\g,\cdot_N)$ which will be used in the following sections.

For $N\in\gl(\g)$ and $S\in\gl(V)$, define $\tilde{\huaL},\tilde{\huaR}:\g\longrightarrow\gl(V)$ by
 \begin{eqnarray}
    \label{eq:rep2}\tilde{\huaL}_x=\huaL_{N(x)}-[\huaL_x,S],\quad  \tilde{\huaR}_x=\huaR_{N(x)}-[\huaR_x,S].
 \end{eqnarray}

 \begin{pro}
 With the above notations. If $(N,S)$ is a Nijenhuis structure on a $\g$-bimodule $V$, then $(V;\tilde{\huaL},\tilde{\huaR})$ is a bimodule over the pre-Lie algebra $(\g,\cdot_N)$.
 \end{pro}
\begin{proof} Since $N+S^*$ is a Nijenhuis operator on the semidirect
product pre-Lie algebra  $\g\ltimes_{\huaL^*-\huaR^*,-\huaR^*} V^*$, for all
$x,y\in \g$ and $\xi,\eta\in V^*$, the deformed operation $\diamond_{N+S^*}$ is
given by
\begin{eqnarray*}
  (x+\xi)\diamond_{N+S^*}(y+\eta)&=&\big((N+S^*)(x+\xi))\cdot_{\huaL,\huaR}(y+\eta)+(x+\xi)\cdot_{\huaL,\huaR}(N+S^*)(y+\eta)\\
  &&-(N+S^*)\big((x+\xi)\cdot_{\huaL,\huaR}(y+\eta)\big)\\
  &=&N(x)\cdot_\g y+x\cdot_\g N(y)-N(x\cdot_\g y)+(\huaL^*_{N(x)}-\huaR^*_{N(x)})\eta+(\huaL^*_{x}-\huaR^*_{x})S^*(\eta)\\
  &&-S^*(\huaL^*_{x}-\huaR^*_{x})(\eta)-\huaR^*_yS^*(\xi)-\huaR^*_{N(y)}\xi+S^*(\huaR^*_y\xi)\\
  &=&x\cdot_N y+(\tilde{\huaL}^*_x-\tilde{\huaR}^*_x)\eta-\tilde{\huaR}^*_y\xi,
\end{eqnarray*}
which implies that $(V^*;\tilde{\huaL}^*-\tilde{\huaR}^*,-\tilde{\huaR}^*)$ is a bimodule over the pre-Lie algebra $(\g,\cdot_N)$. By Proposition \ref{pro:dual-module equiv}, $(V;\tilde{\huaL},\tilde{\huaR})$ is a bimodule over the pre-Lie algebra $(\g,\cdot_N)$.
\end{proof}

\section{$\GRBN$-structures on bimodules over pre-Lie algebras and compatible $\huaO$-operators}\label{sec:ON-structure}

\subsection{$\GRBN$-structures on bimodules over pre-Lie algebras}
In this subsection, we give the notion of an $\GRBN$-structure on a bimodule over a pre-Lie algebra, which is an analogue of the Possion-Nijenhuis structure.

Let  $T:V\longrightarrow \g$ be an $\GRB$-operator and $(N,S)$ a Nijenhuis structure  on a bimodule $(V;\huaL,\huaR)$ over a pre-Lie algebra $(\g,\cdot_\g)$. Recall that $(V,\cdot^T)$ is a pre-Lie algebra, where $\cdot^T$ is given by \begin{equation}\label{eq:pre-Lie operation}
  u\cdot^T v= \huaL_{T(u)}v+\huaR_{T(v)}u,\quad\forall~u,v\in V
\end{equation}
and $T$ is a homomorphism from the pre-Lie algebra $(V,\cdot^T)$ to the pre-Lie algebra $(\g,\cdot_\g)$.

 We define the multiplication $\cdot^T_S:\otimes^2V\longrightarrow V$ to be the deformed multiplication of $\cdot^T$ by $S$, i.e.
\begin{eqnarray}
 \label{eq:defieq21}u\cdot^T_S v&=&S(u)\cdot^T v+u\cdot^T S(v)-S(u\cdot^T v),\quad\forall~u,v\in V.\label{eq:defieq33}
\end{eqnarray}
Define the multiplication $\star^T:\otimes^2V\longrightarrow V$ similar as \eqref{eq:pre-Lie operation} using $( \tilde{\huaL},\tilde{\huaR})$ given by \eqref{eq:rep2}. More precisely,
\begin{eqnarray}
  \label{eq:defieq23} u\star^T v&=&\tilde{\huaL}_{T(u)}v+\tilde{\huaR}_{T(v)}u.
\end{eqnarray}
It is not true in general that the multiplications $\cdot^T_S$ and  $\star^T$ are pre-Lie operations.

\begin{defi}
Let $T:V\longrightarrow \g$ be an $\GRB$-operator and $(N,S)$ a Nijenhuis structure on a $\g$-bimodule $V$.  The triple $(T,N, S)$ is called an {\bf $\GRBN$-structure} on  a $\g$-bimodule $V$  if $T$ and $(N,S)$ satisfy the following conditions
 \begin{eqnarray}
 \label{eq:TN1}N\circ T&=&T\circ S,\\
 \label{eq:TN2} u\cdot^{N\circ T} v&=&u\cdot^{T}_{S} v.
  \end{eqnarray}
  \end{defi}

\begin{lem}\label{lem:mulrel}
Let $(T,N,S)$ be an $\GRBN$-structure. Then we have
 $$ u\cdot^{T}_{S} v=u\star^T v.$$
\end{lem}
\begin{proof} By \eqref{eq:TN1}, we have
$u\cdot^{T}_{S} v+u\star^T v=2u\cdot^{N\circ T} v.$
Then by \eqref{eq:TN2}, we obtain $u\cdot^{T}_{S} v=u\star^T v.$\end{proof}

Thus, if $(T,N,S)$ is an $\GRBN$-structure, then the three multiplications $\cdot^T_S$, $\star^T$ and $\cdot^{N\circ T}$ are the same. Moreover, all of them are pre-Lie operations.

\begin{pro}\label{pro:S-Nijenhuis operator}
 Let $(T,N,S)$ be an $\GRBN$-structure. Then $S$ is a Nijenhuis operator on the pre-Lie algebra $(V,\cdot^T)$. Thus, the multiplications $\cdot^T_S$, $\star^T$ and $\cdot^{N\circ T}$ are pre-Lie operations.
\end{pro}
\begin{proof}  Since $(T,N,S)$ is an $\GRBN$-structure, by the relation $u\cdot^{N\circ T} v=u\cdot^{T}_{S} v$, one has
\begin{eqnarray}
 S(\huaL_{TS(u)}v)+S(\huaR_{T(v)}S(u))&=&\huaL_{TS(u)}S(v)+\huaR_{T(v)}S^2(u);\label{eq:pn11}\\
 S^2(\huaL_{T(u)}v)+S^2(\huaR_{T(v)}u)&=& S(\huaL_{T(u)}S(v))+S(\huaR_{T(v)}S(u)).
\end{eqnarray}
By  \eqref{eq:coNijpair2}, we have
\begin{eqnarray}
 \label{eq:pn22}\huaR_{ T(v)}S^2(u)-\huaR_{TS(v)}S(u)=S(\huaR_{ T(v)}S(u))-S(\huaR_{TS(v)}u).
\end{eqnarray}
By \eqref{eq:pn11}-\eqref{eq:pn22}, we have
\begin{eqnarray*}
  S(u)\cdot^{T}S(v)- S(u\cdot^{T}_{S}v)
  &=&\huaL_{TS(u)}S(v)+\huaR_{T S(v)}S(u)+S^2(\huaL_{T(u)}v)+S^2(\huaR_{T(v)}u)\\
  &&-S(\huaL_{T(u)}S(v))-S(\huaR_{TS(v)}u)-S(\huaL_{TS(u)}v)-S(\huaR_{T(v)}S(u))\\
&=&\huaL_{TS(u)}S(v)+\huaR_{TS(v)}S(u)+S^2(\huaL_{T(u)}v)+S^2(\huaR_{T(v)}u)\\
  &&-S(\huaL_{T(u)}S(v))-S(\huaR_{TS(v)}u)-\huaL_{TS(u)}S(v)-\huaR_{T(v)}S^2(u)\\
&=&\huaL_{TS(u)}S(v)+\huaR_{TS(v)}S(u)+ S(\huaL_{T(u)}S(v))+S(\huaR_{T(v)}S(u))\\
  &&-S(\huaL_{T(u)}S(v))-S(\huaR_{TS(v)}u)-\huaL_{T\circ S(u)}S(v)-\huaR_{T(v)}S^2(u)\\
  &=&\huaR_{TS(v)}S(u)+S(\huaR_{T(v)}S(u))-\huaR_{T(v)}S^2(u)-S(\huaR_{TS(v)}u)\\
  &=&0.
\end{eqnarray*}
Thus, $S$ is a Nijenhuis operator on the pre-Lie algebra $(V,\cdot^T)$.
\end{proof}

\begin{thm}\label{pro:TS1}
 Let $(T,N,S)$ be an $\GRBN$-structure. Then we have
  \begin{itemize}
\item[$\rm(i)$] $T$ is an $\GRB$-operator on the bimodule $(V; \tilde{\huaL},\tilde{\huaR})$ over the deformed pre-Lie algebra $(\g,\cdot_N)$;
 \item[$\rm(ii)$] $N\circ T$ is an $\GRB$-operator on the bimodule  $(V;\huaL,\huaR)$ over the pre-Lie algebra $(\g,\cdot_\g)$.
  \end{itemize}

\end{thm}
\begin{proof} (i)   By Lemma \ref{lem:mulrel}, we have
\begin{eqnarray*}
{T(u\star^Tv)}&=&T(u\cdot^{T}_ {S}v)=T\big(S(u)\cdot^T v+u\cdot^T S(v)-S(u\cdot^T v)\big)\\
&=&T(S(u))\cdot_\g T(v)+T(u)\cdot_\g T(S(v))-TS(u\cdot^T v)\\
&=&N(T(u))\cdot_\g T(v)+T(u)\cdot_\g N(T(v))-N(T(u)\cdot_\g T(v))\\
&=&T(u)\cdot_N T(v).
\end{eqnarray*}
Thus, $T$ is an $\GRB$-operator on the bimodule $(V; \tilde{\huaL},\tilde{\huaR})$ over the deformed pre-Lie algebra $(\g,\cdot_N)$.

(ii) By \eqref{eq:TN2} and the fact that $N$ is a Nijenhuis operator on the pre-Lie algebra $(\g,\cdot_\g)$, we have
\begin{eqnarray*}
 NT(u\cdot^{N\circ T}v)&=&N T(u\cdot^{T}_ {S} v)=N(T(u)\cdot_N T(v))=NT(u)\cdot_\g NT(v),
\end{eqnarray*}
which implies that $N\circ T$ is an $\GRB$-operator on the bimodule  $(V;\huaL,\huaR)$ over the pre-Lie algebra $(\g,\cdot_\g)$.
\end{proof}

\subsection{$\KVN$-structures, $\HN$-structures and $\GRBN$-structures}
In the following, we give the notion of a $\KVN$-structure on a pre-Lie algebra, which consists of an $\frks$-matrix and a Nijenhuis operator on a pre-Lie algebra.  We first recall the notion of an $\frks$-matrix, which plays an important role in the theory of pre-Lie bialgebras.
\begin{defi}{\rm(\cite{Left-symmetric bialgebras})}
Let $(\g,\cdot_\g)$ be a pre-Lie algebra. An element $r\in\Sym^2(\g)$ is called an {\bf $\frks$-matrix} if $\llbracket r,r\rrbracket=0$, where $\llbracket r,r\rrbracket\in\wedge^2\g\otimes \g$ is defined as follows:
\begin{equation}\label{S-equation1}
\llbracket r,r\rrbracket(\xi,\eta,\zeta)=-\langle\xi,r^\sharp(\eta)\cdot_\g r^\sharp(\zeta)\rangle+\langle\eta,r^\sharp(\xi)\cdot_\g r^\sharp(\zeta)\rangle+\langle\zeta,[r^\sharp(\xi),r^\sharp(\eta)]_\g\rangle,\quad\forall~\xi,\eta,\zeta\in\g^*,
\end{equation}
in which $r^\sharp:\g\longrightarrow \g^*$ is a linear map induced by $\langle r^\sharp(\xi),\eta\rangle=r(\xi,\eta)$.
\end{defi}

There is a close relationship between $\frks$-matrices and $\GRB$-operators.
\begin{lem}\label{lem:assrmatrix-GRB}{\rm(\cite{Left-symmetric bialgebras})}
 Let $(\g,\cdot_\g)$ be a pre-Lie algebra.  Then $r\in\Sym^2(\g)$ is an $\frks$-matrix if and only if $r^\sharp$ is an $\GRB$-operator on the bimodule $(\g^*;\ad^*,-R^*)$ over $\g$.
\end{lem}

\begin{defi}\label{defi:rmnij}
Let $r\in\Sym^2(\g)$ be an $\frks$-matrix and $N:\g\longrightarrow \g$ a Nijenhuis operator on a pre-Lie algebra $(\g,\cdot_\g)$.   A pair $(r,N)$ is called a {\bf $\KVN$-structure} on the pre-Lie algebra $(\g,\cdot_\g)$ if for all $\xi,\eta\in \g^*$, the following conditions are satisfied:
\begin{eqnarray}
 \label{eq:rmn1} N\circ r^\sharp&=&r^\sharp\circ N^*,\\
 \label{eq:rmn2} {\xi\cdot^{N\circ r^\sharp}\eta}&=&{\xi\cdot_{N^*}^{r^\sharp}\eta},
\end{eqnarray}
where  \eqref{eq:rmn2} is given by \eqref{eq:TN2} with $T=r^\sharp$, $S=N^*$ and the  bimodule $(\g^*;\ad^*,-R^*)$.
\end{defi}
\begin{rmk}
 The notion of a $\KVN$-structure on the  left-symmetric algebroid was introduced in \cite{WLS}, which combines a Koszul-Vinberg structure and a Nijenhuis operator on a left-symmetric algebroid satisfying some compatibility conditions. A $\KVN$-structure on the left-symmetric algebroid over a point is a $\KVN$-structure on a pre-Lie algebra.
\end{rmk}

By Lemma \ref{lem:assrmatrix-GRB} and the fact that $(N,N^*)$ is a Nijenhuis structure on the pre-Lie algebra $\g$, we have
\begin{pro}\label{thm:rNijRB}
  Let $r\in\Sym^2(\g)$ be an $\frks$-matrix and $N:\g\longrightarrow \g$ a Nijenhuis operator on a pre-Lie algebra $(\g,\cdot_\g)$. Then $(r,N)$ is a $\KVN$-structure on $(\g,\cdot_\g)$ if and only  if $(r^\sharp,N,S=N^*)$ is an $\GRBN$-structure on the bimodule $(\g^*;\ad^*,-R^*)$ over the pre-Lie algebra $(\g,\cdot_\g)$.
\end{pro}

\begin{defi}{\rm(\cite{NiBai})}
 A {\bf pseudo-Hessian  structure} on a pre-Lie algebra $(\g,\cdot_\g)$  is a symmetric nondegenerate $2$-cocycle $\frkB\in\Sym^2(\g^*)$, i.e. $\dt\frkB=0$. More precisely,
$$\frkB(x\cdot_\g y,z)-\frkB(x,y\cdot_\g z)=\frkB(y\cdot_\g x,z)-\frkB(y,x\cdot_\g z),\quad \forall~x,y,z\in\g.$$
\end{defi}

An element $\frkB\in\Sym^2(\g^*)$ induces a linear map $\frkB^\natural:\g\longrightarrow \g^*$ by
$$
\langle\frkB^\natural(x),y\rangle=\frkB(x,y),\quad\forall x,y\in \g.
$$
We say that $\frkB\in\Sym^2(\g^*)$ is nondegenerate if $\frkB^\natural $ is an isomorphism.

\begin{lem}{\rm(\cite{Left-symmetric bialgebras})}\label{lem:conRB}
Let $(\g,\cdot_\g)$ be a pre-Lie algebra. Then a nondegenerate $\frkB\in\Sym^2(\g^*)$ is a pseudo-Hessian  structure on $\g$ if and only if $r\in\Sym^2(\g)$  defined by
  \begin{equation}\label{eq:Connes-r}
r(\xi,\eta)=\langle \xi,(r^\sharp)^{-1}(\eta)\rangle,\quad \forall~\xi,\eta\in \g^*
\end{equation}
is an $\frks$-matrix on $\g$.
\end{lem}

Thus  $\frkB\in\Sym^2(\g^*)$  is  a pseudo-Hessian  structure on $\g$ if and only if $(\frkB^\natural)^{-1}:\g^*\longrightarrow \g$ is an $\GRB$-operator on the bimodule $(\g^*;\ad^*,-R^*)$ over $\g$.

\begin{defi}{\rm(\cite{WBLS})}
Let $\frkB$ be a pseudo-Hessian structure and $N$ a Nijenhuis operator on a pre-Lie algebra $(\g,\cdot_\g)$. Then $(\frkB,N)$ is called a {\bf pseudo-Hessian-Nijenhuis structure} (or {\bf $\HN$-structure}) on the pre-Lie algebra  $(\g,\cdot_\g)$ if for all $x,y\in\g$
\begin{equation}\label{eq:Hess1}
\frkB(N(x),y)=\frkB(x,N(y))
\end{equation} and $\frkB_N:\otimes^2 \g\longrightarrow \g $ defined by $\frkB_N(x,y)=\omega(N(x),y)$ is also $\dt$-closed, i.e.
\begin{equation}\label{eq:ConNijcon}
\frkB(x\cdot_\g y,N(z))-\frkB(x, N(y\cdot_\g z))=\frkB(y\cdot_\g
x,N(z))-\frkB(y,N(x\cdot_\g z)),\quad \forall~x,y,z\in\g.
\end{equation}
\end{defi}

There is a close relationship between  $\HN$-structures and $\GRBN$-structures.
\begin{pro}\label{pro:ConNijRB}
  Let  $\frkB\in\Sym^2(\g^*)$ be a nondegenerate bilinear form and $N:\g\longrightarrow \g$ a Nijenhuis operator on a  pre-Lie algebra $(\g,\cdot_\g)$. Then $(\frkB,N)$ is an $\HN$-structure on $(\g,\cdot_\g)$ if and only  if $((\frkB^\natural)^{-1},N,S=N^*)$ is an $\GRBN$-structure on the bimodule $(\g^*;\ad^*,-R^*)$ over the pre-Lie algebra $(\g,\cdot_\g)$.
\end{pro}
\begin{proof}  It is obvious that \eqref{eq:Hess1} is equivalent to that  $(\frkB^\natural)^{-1}\circ N^*=N\circ (\frkB^\natural)^{-1}$. By Lemma \ref{lem:conRB}, $\frkB\in\Sym^2(\g^*)$ is a pseudo-Hessian structure if and only if $(\frkB^\natural)^{-1} $ is an $\GRB$-operator on the bimodule $(\g^*;\ad^*,-R^*)$ over $\g$. Thus we only need to show that \eqref{eq:ConNijcon} is equivalent to
\begin{equation}\label{eq:ttt}\xi\cdot^{N\circ (\frkB^\natural)^{-1}}\eta= \xi\cdot^{(\frkB^\natural)^{-1}}_{N^*}\eta,\quad \forall~ \xi,\eta\in \g^*.
\end{equation}
 Assume that $\xi=\frkB^\natural(x)$ and $\eta=\frkB^\natural(y)$ for some $x,y\in \g$. For all $z\in \g$, by \eqref{eq:Hess1} and the fact that $\frkB$ is $\dt$-closed, we have
\begin{eqnarray*}
 \langle\xi\cdot^{N\circ (\frkB^\natural)^{-1}}\eta-\xi\cdot^{(\frkB^\natural)^{-1}}_{N^*}\eta,z\rangle
 &=&\frkB(x\cdot_\g z,N(y))-\frkB(x, N(z\cdot_\g y))-\frkB(z\cdot_\g
x,N(y))+\frkB(z,N(x\cdot_\g y)),
\end{eqnarray*}
which implies that \eqref{eq:ConNijcon} is equivalent to \eqref{eq:ttt}.
\end{proof}

 By Proposition \ref{thm:rNijRB} and Proposition \ref{pro:ConNijRB}, we have
\begin{cor}\label{cor:Connes-r-matrix-Nij}
  Let   $\frkB\in\Sym^2(\g^*)$ be a nondegenerate bilinear form and $N:\g\longrightarrow \g$ a Nijenhuis operator on a pre-Lie algebra $(\g,\cdot_\g)$. Then $(\frkB,N)$ is an $\HN$-structure on $(\g,\cdot_\g)$ if and only  if $(r,N)$ is a $\KVN$-structure on $(\g,\cdot_\g)$, where $r\in\Sym^2(\g)$ is given by \eqref{eq:Connes-r}.
\end{cor}

\begin{ex}
   Let $(\g,\cdot_\g)$ be a 2-dimensional pre-Lie algebra with a basis
$\{e_1,e_2\}$ whose non-zero products are given as follows:
$$e_2\cdot_\g e_1=-e_1,\quad e_2\cdot_\g e_2=e_2.$$
  Let $\{e_1^*,e_2^*\}$ be the dual basis. Then $(\frkB,N)$ given by
  $$\frkB=a e_1^*\otimes e_2^*+ a e_2^*\otimes e_1^*+b e_2^*\otimes e_2^*,~  a\neq 0,\quad N =\begin{bmatrix}c&
d\\ 0&c\end{bmatrix} $$
 is an $\HN$-structure on the pre-Lie algebra $(\g,\cdot_\g)$.

 Therefore, $(r,N)$ given by
  $$r=-\frac{b}{a^2} e_1\otimes e_1+\frac{1}{a} e_1\otimes e_2+ \frac{1}{a} e_2\otimes e_1,~  a\neq 0,\quad N =\begin{bmatrix}c&
d\\ 0&c\end{bmatrix} $$
 is a $\KVN$-structure on the pre-Lie algebra $(\g,\cdot_\g)$.
\end{ex}

\begin{ex}\label{ex:HN2}
    Let $(\g,\cdot_\g)$ be a 3-dimensional pre-Lie algebra
with a basis $\{e_1,e_2,e_3\}$ whose non-zero products are given
as follows:
$$e_3\cdot_\g e_2=e_2,\quad e_3\cdot_\g e_3=-e_3.$$
  Let $\{e_1^*,e_2^*,e_3^*\}$ be the dual basis. Then $(\frkB,N)$ given by
  $$\frkB=a e_1^*\otimes e_1^*+ b e_2^*\otimes e_3^*+b e_3^*\otimes e_2^*+ce_3^*\otimes e_3^*,\quad ab\neq 0,\quad N =\begin{bmatrix}d&
0&0\\ 0&e&f\\0&0&e\end{bmatrix}$$
 is an $\HN$-structure on the pre-Lie algebra $(\g,\cdot_\g)$.

 Therefore, $(r,N)$ given by
  $$r=\frac{1}{a} e_1\otimes e_1-\frac{c}{b^2} e_2\otimes e_2-b e_2\otimes e_3-b e_3\otimes e_2,\quad ab\neq 0,\quad N =\begin{bmatrix}d&
0&0\\ 0&e&f\\0&0&e\end{bmatrix}$$
 is a $\KVN$-structure on the pre-Lie algebra $(\g,\cdot_\g)$.
\end{ex}

\subsection{Compatible $\GRB$-operators and $\GRBN$-structures}
First, we recall the notion of compatible $\huaO$-operators.
\begin{defi}
Let $(\g,\cdot_\g)$ be a pre-Lie algebra and $(V;\huaL,\huaR)$  a
bimodule over $\g$. Let $T_1,T_2: V\longrightarrow \g$ be two
$\huaO$-operators on the bimodule $(V;\huaL,\huaR)$. If for any
$k_1,k_2$, $k_1T_1+k_2T_2$ is still an $\huaO$-operator, then $T_1$ and $T_2$ are called {\bf
compatible}.
\end{defi}

A pair of compatible $\huaO$-operators can  give rise to a Nijenhuis operator under some conditions.
\begin{pro}{\rm(\cite{WBLS})}\label{thm:comRBN}
Let $T_1,T_2: V\longrightarrow \g$ be two $\huaO$-operators on
a  bimodule $(V;\huaL,\huaR)$ over a pre-Lie algebra $(\g,\cdot_\g)$. Suppose that $T_2$ is invertible. If $T_1$ and $T_2$
are compatible, then $N=T_1\circ T_2^{-1}$ is a Nijenhuis operator on the pre-Lie algebra $(\g,\cdot_\g)$.
\end{pro}

Furthermore, we have
\begin{pro}{\rm(\cite{WBLS})}
Let $T_1,T_2: V\longrightarrow \g$ be two $\huaO$-operators on a bimodule $(V;\huaL,\huaR)$ over a pre-Lie algebra $(\g,\cdot_\g)$. Suppose that  $T_1$ and $T_2$ are invertible. Then
$T_1$ and $T_2$ are compatible if and only if $N=T_1\circ T_2^{-1}$ is a
Nijenhuis operator.
\end{pro}

In the following, we construct compatible $\GRB$-operators
from  an $\GRBN$-structure. Given an $\GRBN$-structure $(T,N,S)$ on a $\g$-bimodule $V$, by Theorem \ref{pro:TS1}, $T\circ S$ is an $\GRB$-operator.  In
fact, $T$ and $T\circ S$ are compatible.
\begin{pro}\label{pro:TS2}
 Let $(T,N,S)$ be an $\GRBN$-structure on a bimodule  $(V;\huaL,\huaR)$ over a pre-Lie algebra $(\g,\cdot_\g)$. Then $T$ and $T\circ S=N\circ T$ are compatible $\GRB$-operators.
\end{pro}
\begin{proof} It is sufficient to prove that $T+T\circ S$ is an $\GRB$-operator.
Since  $(T,N,S)$ is an $\GRBN$-structure, we have
$$u\cdot^{T+T\circ S}v=u\cdot^T v+u\cdot^{T\circ S}v=u\cdot^T v+u\cdot^T_S v.$$
Thus, we have
\begin{eqnarray*}
  (T+T\circ S)(u\cdot^{T+T\circ S}v)
  &=&T(u\cdot^T v)+ T(u\cdot^T_S v)+T\circ S(u\cdot^{T\circ S} v)+T\circ S(u\cdot^T v)\\
  &=&T(u\cdot^T v)+T\circ S(u\cdot^{T\circ S} v)+T\circ S(u\cdot^T v)\\
  &&+T\big(S(u)\cdot^T v+u\cdot^T S(v)-S(u\cdot^T v)\big)\\
  &=&T(u\cdot^T v)+T\circ S(u\cdot^{T\circ S} v)+T\big(S(u)\cdot^T v+u\cdot^T S(v)\big)\\
  &=&T(u)\cdot_\g T(v)+T( S(u))\cdot_\g T(S(v))+T(S(u))\cdot_\g T(v)+T(u)\cdot_\g T(S(v))\\
  &=&(T+T\circ S)(u)\cdot_\g(T+T\circ S)(v),
\end{eqnarray*}
which means that $T+T\circ S$ is an $\GRB$-operator.
\end{proof}

Compatible $\GRB$-operators  can give rise to $\GRBN$-structures.

\begin{pro}\label{pro:ComptoNS}
Let $T_1, T_2: V\longrightarrow \g$ be two $\GRB$-operators on a $\g$-bimodule $V$. Suppose that $T_2$ is invertible. If
$T_1$ and $T_2$ are compatible, then
 $(T_i,N=T_1\circ T_2^{-1},S=T_2^{-1}\circ T_1)$ is an $\GRBN$-structure, for $i=1,2$.
 \end{pro}
\begin{proof}  First we prove that $(N=T_1\circ T_2^{-1},S=T_2^{-1}\circ T_1)$ is a Nijenhuis structure on the $\g$-bimodule $V$. By Proposition \ref{thm:comRBN}, $N=T_1\circ T_2^{-1}$ is a Nijenhuis operator on the pre-Lie algebra $(\g,\cdot_\g)$. Since $T_1$ and $T_2$ are compatible, for all~$u,v\in V$, we  have
\begin{eqnarray*}
  T_2(u)\cdot_\g T_1(v)+ T_1(u)\cdot_\g
 T_2(v)=T_2\big(\huaL_{T_1(u)}v+\huaR_{T_1(v)}u\big) +T_1\big(\huaL_{T_2(u)}v+\huaR_{T_2(v)}u\big).
\end{eqnarray*}
Substituting $T_1$ by $T_2\circ S$, then we have
\begin{eqnarray*}
  T_2(u)\cdot_\g T_2(S(v))+ T_2(S(u))\cdot_\g
 T_2(v)=T_2\big(\huaL_{T_2(S(u))}v+\huaR_{T_2(S(v))}u\big) +T_2S\big(\huaL_{T_2(u)}v+\huaR_{T_2(v)}u\big).
\end{eqnarray*}
Since $T_2$ is an $\GRB$-operator, we have
\begin{eqnarray*}
 T_2(u)\cdot_\g T_2(S(v))+ T_2(S(u))\cdot_\g
 T_2(v)=T_2\big(\huaL_{T_2(u)}S(v)+\huaR_{T_2(S(v))}u+\huaL_{T_2(S(u))}v+\huaR_{T_2(v)}S(u)\big).
\end{eqnarray*}
Since $T_2$ is invertible, we obtain
\begin{eqnarray}\label{eq:ON3}
  S\big(\huaL_{T_2(u)}v+\huaR_{T_2(v)}u\big)=\huaL_{T_2(u)}S(v)+\huaR_{T_2(v)}S(u).
\end{eqnarray}
Therefore, we have
\begin{eqnarray}
\label{eq:TNS7}\huaL_{T_2(u)}S^2(v)-S\big(\huaL_{T_2(u)}S(v)\big)&=&S\big(\huaR_{T_2S(v)}u\big)-\huaR_{T_2S(v)}S(u),\\
 \label{eq:TNS8}S\big(\huaL_{T_2 S(u)}v\big)-\huaL_{T_2S(u)}S(v)&=&\huaR_{T_2(v)}S^2(u)-S\big(\huaR_{T_2(v)}S(u)\big).
\end{eqnarray}
Furthermore, since $T_2$ and $T_2\circ S$ are $\GRB$-operators, we have
$$T_2\circ S(u\cdot^{T_2\circ S}v)=T_2( S(u))\cdot_\g T_2(S(v))=T_2(S(u)\cdot^{T_2} S(v)).$$
Since $T_2$ is invertible, we have
\begin{equation}\label{eq:ON4}
S(u\cdot^{T_2\circ S}v)=S(u)\cdot^{T_2} S(v),
\end{equation}
which means that
\begin{eqnarray}\label{eq:TSN6}
  S\Big(\huaL_{T_2S(u)}v+\huaR_{T_2S(v)}u\Big)=\huaL_{T_2 S(u)}S(v)+\huaR_{T_2 S(v)}S(u).
\end{eqnarray}
By  \eqref{eq:TNS7} and \eqref{eq:TSN6}, one has
\begin{eqnarray*}
 0&=&\huaL_{T_2(u)}S^2(v)+ S\big(\huaL_{T_2S(u)}v\big)-S\big(\huaL_{T_2(u)}S(v)\big)-\huaL_{T_2S(u)}S(v)\\
 &=&\huaL_{T_2(u)}S^2(v)+ S\big(\huaL_{NT_2(u)}v\big)-S\big(\huaL_{T_2(u)}S(v)\big)-\huaL_{NT_2(u)}S(v).
\end{eqnarray*}
Since $T_2$ is invertible and let $x=T_2(u)$, we have
$$S\big(\huaL_xS(v)\big)-\huaL_xS^2(v)- S\big(\huaL_{N(x)}v\big)+\huaL_{N(x)}S(v)=0,$$
which implies that \eqref{eq:coNijpair1} holds

Similarly, by \eqref{eq:TNS8} and \eqref{eq:TSN6}, one has
\begin{eqnarray*}
 0&=&\huaR_{T_2(v)}S^2(u)+ S\big(\huaR_{T_2S(v)}u\big)-S\big(\huaR_{T_2(v)}S(u)\big)-\huaR_{T_2S(v)}S(u)\\
 &=&\huaR_{T_2(v)}S^2(u)+ S\big(\huaR_{N T_2(v)}u\big)-S\big(\huaR_{T_2(v)}S(u)\big)-\huaR_{NT_2(v)}S(u).
\end{eqnarray*}
Since $T_2$ is invertible and let $x=T_2(v)$, we have
$$S\big(\huaR_xS(u)\big)-\huaR_xS^2(u)- S\big(\huaR_{N(x)}u\big)+\huaR_{N(x)}S(u)=0,$$
which implies that \eqref{eq:coNijpair2} holds. Thus  $(N,S)$ is a Nijenhuis structure.

(i) It is obvious that $T_2\circ S=N\circ T_2=T_1$.
By \eqref{eq:ON3}, we have
\begin{eqnarray*}
{u\cdot^{T_2}_{S}v}-{u\cdot^{T_2\circ S}v}=\huaL_{T_2(u)}S(v)+\huaR_{T_2(v)}S(u)-S\big(\huaL_{T_2(u)}v+\huaR_{T_2(v)}u\big)=0,
\end{eqnarray*}
which  implies that $ {u\cdot^{T_2}_{S}v}={u\cdot^{T_2\circ S}v}$.
Therefore, $(T_2,N=T_1\circ T_2^{-1},S=T_2^{-1}\circ T_1)$ is an $\GRBN$-structure.

(ii) By \eqref{eq:ON4}, we have
\begin{eqnarray*}
 {u\cdot^{T_1}_{S}v}-{u\cdot^{T_1\circ S}v}&=&\huaL_{T_1(u)}S(v)+\huaR_{T_1(v)}S(u)-S\big(\huaL_{T_1(u)}v+\huaR_{T_1(v)}u\big)\\
  &=&\huaL_{T_2(S(u))}S(v)+\huaR_{T_2(S(v))}S(u)-S\big(\huaL_{T_2(S(u))}v+\huaR_{T_2(S(v))}u\big)\\
  &=&S(u)\cdot^{T_2} S(v)-S(u\cdot^{T_2\circ S}v)=0.
\end{eqnarray*}
Thus, $(T_1,N=T_1\circ T_2^{-1},S=T_2^{-1}\circ T_1)$ is  an $\GRBN$-structure.
\end{proof}

 In the sequel, we show that an $\GRBN$-structure on
a $\g$-bimodule $V$ gives rise to a hierarchy of compatible $\GRB$-operators.

\begin{lem}
 Let $(T,N,S)$ be an $\GRBN$-structure on
a $\g$-bimodule $V$.  Then for all $k,i\in\Nat$, we have
\begin{eqnarray}
\label{eq:TSN3}T_k(u\cdot^T_{S^{k+i}}v)&=& T_k(u)\cdot_{N^i}T_k(v);\\
\label{eq:TSN4}  u\cdot^{T_k}v&=& u\cdot^{T}_{S^k}v=S^{k-i}(u\cdot^{T_i}v),
\end{eqnarray}
where $T_k=T\circ S^k=N^k\circ T$ and set $T_0=T$.
\end{lem}

\begin{thm}\label{thm:hierarchy}
Let $(T,N,S)$ be an $\GRBN$-structure on a $\g$-bimodule $V$. Then all $T_k=N^k\circ T$ are $\GRB$-operators. Furthermore, for all $k,l\in\Nat$, $T_k$ and $T_l$ are compatible.
\end{thm}
\begin{proof} By  \eqref{eq:TSN3} and \eqref{eq:TSN4} with $i=0$, we have
\begin{eqnarray*}
 T_k(u\cdot^{T_{k}}v)= T_k(u)\cdot_\g T_k(v),
\end{eqnarray*}
which implies that $T_k$ is an $\GRB$-operator.
For the second conclusion, we need  to prove that $T_k+T_{k+i}$ are $\GRB$-operators for all $k,i\in\Nat$.
By  \eqref{eq:TSN4}, we have
$$u\cdot^{T_k+T_{k+i}}v=u\cdot^{T_k}v+u\cdot^{T_{k+i}}v=u\cdot^{T_k}v+u\cdot^{T_k}_{S^i}v.$$
Thus, we have
\begin{eqnarray*}
  (T_k+T_{k+i})(u\cdot^{T_k+T_{k+i}}v)
  &=&T_k(u\cdot^{T_k}v)+T_k(u\cdot^{T_k}_{S^i}v)+T_{k+i}(u\cdot^{T_k}v)+T_{k+i}(u\cdot^{T_k}_{S^i}v)\\
  &=&T_k(u\cdot^{T_k}v)+T_{k+i}(u\cdot^{T_k}v)+T_{k+i}(u\cdot^{T_k}_{S^i}v)\\
  &&+T_k\big(S^i(u)\cdot^{T_k}v+u\cdot^{T_k}S^i(v)-S^i(u\cdot^{T_k}v)\big)\\
  &=&T_k(u\cdot^{T_k}v)+T_{k+i}(u\cdot^{T_k}_{S^i}v)+T_k\big(S^i(u)\cdot^{T_k}v+u\cdot^{T_k}S^i(v)\big)\\
  &=&T_k(u)\cdot_\g T_k(v)+T_{k+i}(u)\cdot_\g T_{k+i}(v)+T_{k+i}(u)\cdot_\g T_k(v)+T_k(u)\cdot_\g T_{k+i}(v)\\
  &=&(T_k+T_{k+i})(u)\cdot_\g (T_k+T_{k+i})(v),
\end{eqnarray*}
which implies that $T_k+T_{k+i}$ are $\GRB$-operators.
\end{proof}

By Theorem \ref{thm:hierarchy} and Proposition \ref{thm:rNijRB}, we have
\begin{cor}
  Let $(r,N)$ be a $\KVN$-structure on pre-Lie algebra $(\g,\cdot_\g)$. Then for all $k\in\Nat$, $r_k\in\Sym^2(\g)$ defined by $r_k(\xi,\eta)=\langle N^k r^\sharp(\xi),\eta\rangle$ for all $\xi,\eta\in \g^*$, are pairwise compatible $\frks$-matrices in the sense that any linear combination of $r_k$ and $  r_l$ is still an $\frks$-matrix.
\end{cor}

\begin{pro}\label{pro:GRBN-invertible}
 Let $(T,N,S)$ be an $\GRBN$-structure on a $\g$-bimodule $V$. If $T$ is invertible, then, for all $k\in\Nat$,
  $(T,N^k,S^k)$ and $(T_k=N^k\circ T,N^k,S^k)$ are $\GRBN$-structures.
\end{pro}
\begin{proof} Since $(T,N,S)$ is an $\GRBN$-structure on a $\g$-bimodule $V$, by Theorem \ref{thm:hierarchy},  $T$ and $T_k=N^k\circ T$ are compatible $\GRB$-operators.  Then by the condition that $T$ is invertible and Proposition \ref{pro:ComptoNS}, the conclusions follow immediately.
\end{proof}

By Proposition \ref{pro:ConNijRB}, we have
\begin{cor}
 Let $(\frkB,N)$ be an $\HN$-structure on a pre-Lie algebra $(\g,\cdot_\g)$. Then for all $k\in\Nat$,
  $(\frkB,N^k)$ are $\HN$-structures and $(r_k,N^k)$ are $\KVN$-structures, where $r_k\in\Sym^2(\g)$ is defined by $r_k(\xi,\eta)=\langle N^k (\frkB^\natural)^{-1}(\xi),\eta\rangle$.
\end{cor}

\section{Solutions of the strong Maurer-Cartan equation and $\GRBN$-structures}\label{sec:MC-ON}
\subsection{Solutions of the strong Maurer-Cartan equation on $\g\bowtie V_T $ }

Let $T: V\longrightarrow \g$ be an $\GRB$-operator on a
bimodule $(V;\huaL,\huaR)$ over a pre-Lie algebra  $(\g,\cdot_\g)$. We denote by  $V_T:=(V,\cdot^T)$ the pre-Lie algebra given by \eqref{eq:pre-Lie operation}. Define $\frkL^T,\frkR^T:V\longrightarrow\gl(\g)$ by
\begin{eqnarray}
  \label{eq:HT1}\frkL^T_ux=T(u)\cdot_\g x-T(\huaR_x u),\quad
  \frkR^T_ux=x\cdot_\g T(u)-T(\huaL_x u),\quad \forall x\in \g, u\in V.
\end{eqnarray}

\begin{thm}{\rm(\cite{Liu19})}\label{thm:Uchino}
With the above notations, $(\g; \frkL^T, \frkR^T)$ is a bimodule over the pre-Lie algebra $(V,\cdot^T)$  and $(\g\oplus V,\diamond_T)$ is a pre-Lie algebra, where the multiplication $\diamond_T$ is given by
$$
 (x+u)\diamond_T(y+v)=x\cdot_\g y+\frkL^T_uy+\frkR^T_v x+\huaL_xv+\huaR_yu+u\cdot^Tv,\quad\forall x,y\in A, u,v\in V.
$$

We denote the corresponding twilled pre-Lie algebra by $\g\bowtie V_T$.
 \end{thm}

 By Lemma \ref{lem:SCM-condition}, we have
\begin{pro}\label{pro:SMC-O-operator}
Let $T: V\longrightarrow \g$ be an $\GRB$-operator on a
bimodule $(V;\huaL,\huaR)$ over a pre-Lie algebra  $(\g,\cdot_\g)$. A linear operator $\Omega:\g\rightarrow V$ is a solution of the strong Maurer-Cartan equation on the  twilled pre-Lie algebra $\g\bowtie V_T$ if and only if
\begin{eqnarray}
  \Omega(x\cdot_\g y)&=&\huaL_{x}\Omega(y)+\huaR_{y}\Omega(x),\label{eq:SMC-1}\\
  \Omega(x)\cdot^T\Omega(y)&=&\Omega\big(\frkL^T_{  \Omega(x)}y+\frkR^T_{  \Omega(y)}x\big).\label{eq:SMC-2}
\end{eqnarray}
\end{pro}

By \eqref{eq:SMC-2}, it is obvious that if $\Omega:\g\longrightarrow V_T$ is a solution of the strong Maurer-Cartan equation on $\g \bowtie V_T$, then $\Omega$ is an $\GRB$-operator on the bimodule
$(\g;\frkL^T,\frkR^T)$ over the pre-Lie algebra $(V,\cdot^T)$. The $\GRB$-operator $\Omega$ induces a pre-Lie algebra $\g_\Omega=(\g,\cdot^\Omega)$, where $\cdot^\Omega$ is given by
$$
x\cdot^\Omega y=\frkL^T_{  \Omega(x)}y+\frkR^T_{  \Omega(y)}x,\quad \forall x,y\in \g.
$$
Furthermore, according to Theorem \ref{thm:Uchino}, $\huaL^\Omega,
\huaR^\Omega:\g\longrightarrow\gl(V)$ defined by $$
\huaL^\Omega_xu=\Omega(x)\cdot^Tu-\Omega(\frkR^T_ux),\;\;
\huaR^\Omega_xu=u\cdot^T\Omega(x)-\Omega(\frkL_u^Tx),\quad\forall
x\in \g, u\in V$$  gives a
bimodule structure of $
\g_\Omega$ on $V$. On $\g\oplus V$, define a
multiplication $\ast_T^\Omega$ by
\begin{eqnarray*}
  (x+u)\ast_T^\Omega(y+v)=x\cdot^\Omega y+\frkL^T_uy+\frkR^T_v x+\huaL^\Omega_xv+\huaR^\Omega_yu +u\cdot^Tv,\quad\forall x,y\in \g, u,v\in V.
\end{eqnarray*}
\begin{pro} \label{pro:SMC-Hirac1}
Let $\Omega:\g\longrightarrow V_T$ be a solution of the strong Maurer-Cartan equation on $\g\bowtie V_T$. Then we have
\begin{itemize}
  \item[\rm(i)] $(\g\oplus V,\ast_T^\Omega)$ is a pre-Lie algebra and we denote the corresponding twilled algebra by $\g_\Omega\bowtie V_T$;
   \item[\rm(ii)] $T$ is a solution of the strong Maurer-Cartan equation on the  twilled algebra $\g_\Omega\bowtie V_T$;
    \item[\rm(iii)] $T$ is an $\GRB$-operator on the bimodule $(V;\huaL^\Omega,\huaR^\Omega)$ over the pre-Lie algebra  $(\g,\cdot^\Omega)$.
\end{itemize}
\end{pro}

\begin{proof}  (i) By Theorem \ref{thm:Uchino}, it is straightforward to deduce that $(\g\oplus V,\ast_T^\Omega)$ is a pre-Lie algebra.

(ii) To prove that $T$ is a solution of the strong Maurer-Cartan equation on the  twilled pre-Lie algebra $\g_\Omega\bowtie V_T$, we only need to verify the following two equalities:
\begin{eqnarray*}
 T(u\cdot^T v)=\frkL^T_{u}T(v)+\frkR^T_{v}T(u),\quad
  T(u)\cdot^\Omega T(v)=T\big(\huaL^\Omega_{ T(u)}v+\huaR^\Omega_{  T(v)}u\big).
\end{eqnarray*}
The first equality follows from that $T$ is an $\GRB$-operator on a bimodule
$(V;\huaL,\huaR)$ over $\g$. For the second equality, on the one hand, by \eqref{eq:SMC-1}, we have
$$ T(u)\cdot^\Omega T(v)=(T\Omega T(u))\cdot_\g T(v)+ T (u)\cdot_\g( T\Omega T(v))-T\Omega(T(u)\cdot_\g T(v));$$
on the other hand, by the fact that $T$ is an $\GRB$-operator on a bimodule
$(V;\huaL,\huaR)$ over $\g$, we have
\begin{eqnarray*}
T(\huaL^\Omega_{ T(u)}v+\huaR^\Omega_{  T(v)}u)&=&T\big((\Omega T(u))\cdot^T v+u\cdot^T (\Omega T(v))-\Omega(T(u)\cdot_\g T(v))\big)\\
&=&(T\Omega T(u))\cdot_\g T(v)+ T (u)\cdot_\g( T\Omega T(v))-T\Omega(T(u)\cdot_\g T(v)),
\end{eqnarray*}
which implies that the second equality also holds.

(iii) can be obtained by a straightforward calculation. We omit the details.
\end{proof}

\begin{pro}\label{pro:RB-SMC}
  Let $R$ be a Rota-Baxter operator on the pre-Lie algebra $(\g,\cdot_\g)$, that is,
  $$R(x)\cdot_\g R(y)= R(R(x)\cdot_\g y+x\cdot_\g R(y)),\quad\forall~x,y\in\g.$$
  Then $\Omega:\g\longrightarrow \g$ is a solution of the strong Maurer-Cartan equation on $\g_R\bowtie\g $ if and only if
  \begin{eqnarray}
  \label{eq:RB-MC1}  \Omega(x\cdot_\g y)&=&\Omega(x)\cdot_\g y+x\cdot_\g \Omega(y),\\
   \label{eq:RB-MC2}  \Omega R \Omega(x\cdot_\g y)&=&\Omega R \Omega(x)\cdot_\g y+x\cdot_\g \Omega R \Omega(y),
  \end{eqnarray}
  that is, $\Omega$ and $\Omega R \Omega$ are derivations of the pre-Lie algebra $\g$.
\end{pro}
\begin{proof}
  Since the Rota-Baxter operator $R: \g\longrightarrow \g$ is an $\huaO$-operator on the bimodule $(\g;L,R)$ over $\g$, it is obvious that \eqref{eq:RB-MC1} and \eqref{eq:SMC-1} are equivalent. By a direct calculation, the right hand side of \eqref{eq:SMC-2} is equal to
  $$\Omega\Big(R\Omega(x)\cdot_\g y-R(\Omega(x)\cdot_\g y)+x\cdot_\g R\Omega(y)-R(x\cdot_\g \Omega(y))\Big).$$
  By \eqref{eq:RB-MC1},
  \begin{eqnarray*}
    &&\Omega\Big(R\Omega(x)\cdot_\g y-R(\Omega(x)\cdot_\g y)+x\cdot_\g R\Omega(y)-R(x\cdot_\g \Omega(y))\Big)\\
    &=&\Omega R \Omega(x)\cdot_\g y+x\cdot_\g \Omega R \Omega(y)-\Omega R \Omega(x\cdot_\g y)+R\Omega(x)\cdot_\g\Omega(y)+\Omega(x)\cdot_\g R\Omega(y).
  \end{eqnarray*}
  On the other hand, the left hand side of \eqref{eq:SMC-2} is equal to
  $$R\Omega(x)\cdot_\g\Omega(y)+\Omega(x)\cdot_\g R\Omega(y).$$
  Comparing the both sides of \eqref{eq:SMC-2}, \eqref{eq:RB-MC2} follows immediately. The converse can be proved similarly.
\end{proof}

\emptycomment{\begin{ex}\label{ex:RB-MC1}
  We put $\g=C^1([0,1])$. Then $\g$ is a communicative associative algebra with the usual multiplication and thus $\g$ is a pre-Lie algebra. The integral operator is a Rota-Baxter operator:
\begin{equation*}
   R:\g\rightarrow \g,\quad R(f)(x):=\int^x_0f(t)dt.
\end{equation*}
 Define $\Omega:\g\rightarrow \g$ by
  \begin{equation*}
  \Omega(f)(x):=\kappa(x)\frac{df}{dx}(x)=\kappa(x)f'(x),\quad \kappa(x)\in \g.
  \end{equation*}
By direct calculation, we have
$$\Omega R \Omega(f)(x)=\kappa^2(x)f'(x).$$
It is obvious that $\Omega$ and $\Omega R \Omega$ are derivations of the pre-Lie algebra $\g$. Thus $\Omega$ is a solution of the strong Maurer-Cartan equation on $\g_R\bowtie\g $.
\end{ex}}

\begin{ex}\label{ex:RB-MC11}
  We put $\g=C^{\infty}([0,1])$. Define the multiplication
  $$f\cdot_\g g:=f(x)\frac{dg}{dx}(x)=f(x)g'(x),\quad\forall~f,g\in\g.$$
  Then $(\g,\cdot_\g)$ is a pre-Lie algebra. The integral operator is a Rota-Baxter operator:
\begin{equation*}
   R:\g\rightarrow \g,\quad R(f)(x):=\int^x_0f(t)dt.
\end{equation*}
 Define $\Omega:\g\rightarrow \g$ by
  \begin{equation*}
  \Omega(f)(x):=\lambda\frac{df}{dx}(x)=\lambda f'(x),\quad \lambda\in \K.
  \end{equation*}
By direct calculation, we have
$$\Omega R \Omega(f)(x)=\lambda^2f'(x).$$
It is obvious that $\Omega$ and $\Omega R \Omega$ are derivations of the pre-Lie algebra $\g$. Thus $\Omega$ is a solution of the strong Maurer-Cartan equation on $\g_R\bowtie\g $.
\end{ex}

\begin{ex}\label{ex:RB-MC2}
Let $(\g,\cdot_\g)$ be a 3-dimensional pre-Lie algebra
with a basis $\{e_1,e_2,e_3\}$ given
as follows:
$$e_1\cdot_\g e_1=e_1,\quad e_1\cdot_\g e_2=e_3,$$
whose sub-adjacent Lie algebra is the $3$-dimensional Heisenberg Lie algebra. Let $R(e_j)=\sum_{i=1}^3 r_{ij}e_i$ and $\Omega(e_j)=\sum_{i=1}^3 \omega_{ij}e_i$.  Then $\Omega$ is a solution of the strong Maurer-Cartan equation on $\g\bowtie\g_R $, in which the Rota-Baxter operator $R$ and $\Omega$ are given as follows:
\begin{itemize}
  \item[\rm(1)] $R =\begin{bmatrix}0&
0&0\\ r_{21}&r_{22}&0\\ r_{31}& r_{32}& 0\end{bmatrix},\quad \Omega=\begin{bmatrix}0&
0&0\\ 0&0&0\\ 0& \omega_{32}& 0\end{bmatrix};$
   \item[\rm(2)] $R =\begin{bmatrix}0&
0&0\\ 0&0&0\\ r_{31}& r_{32}& r_{33}\end{bmatrix}(r_{33}\neq 0),\quad \Omega=\begin{bmatrix}0&
0&0\\ 0&0&0\\ 0& \omega_{32}& 0\end{bmatrix};$
    \item[\rm(3)] $R =\begin{bmatrix}0&
0&0\\ r_{21}&0&0\\ r_{31}& r_{32}& 0\end{bmatrix},\quad \Omega=\begin{bmatrix}0&
0&0\\ 0&\omega_{22}&0\\0& \omega_{32}& \omega_{22}\end{bmatrix}$.
\end{itemize}

\end{ex}
\subsection{Solutions of the strong Maurer-Cartan equation and $\GRBN$-structures}
In the following, we study the relationships between solutions of the strong Maurer-Cartan equations and $\GRBN$-structures. First, we have
\begin{pro}\label{pro:MC-Nijenhuis}
  Let $\Omega:\g\longrightarrow V_T$ be a solution of the strong Maurer-Cartan equation on $\g\bowtie V_T$. Then $N=T\circ\Omega$ is a Nijenhuis operator on the pre-Lie algebra $(\g,\cdot_\g)$.
\end{pro}
\begin{proof}
Applying $T$ to \eqref{eq:SMC-2}, we have
  $$T\Omega(x)\cdot_{\g}T\Omega(y)=T\Omega\big(\frkL^T_{  \Omega(x)}y+\frkR^T_{  \Omega(y)}x\big).$$
  On the right-hand side,
  $$\frkL^T_{  \Omega(x)}y+\frkR^T_{  \Omega(y)}x=T\Omega(x)\cdot_{\g}y-T(\huaR_y \Omega(x))+x\cdot_{\g}T\Omega(y)-T(\huaL_x \Omega(y)).$$
 By \eqref{eq:SMC-1}, we have
 $$\frkL^T_{  \Omega(x)}y+\frkR^T_{  \Omega(y)}x=T\Omega(x)\cdot_{\g}y+x\cdot_{\g}T\Omega(y)-T\Omega(x\cdot_\g y).$$
 Thus we have
 $$T\Omega(x)\cdot_{\g}T\Omega(y)=T\Omega\big(T\Omega(x)\cdot_{\g}y+x\cdot_{\g}T\Omega(y)-T\Omega(x\cdot_\g y)\big).$$
  \end{proof}

Now we give the main result in this section, which says that a solution of the strong Maurer-Cartan equation on the twilled pre-Lie algebra associated to an $\GRB$-operator gives rise to an $\GRBN$-structure.

\begin{thm}\label{thm:MC-GRBN}
 Let $T: V\longrightarrow \g$ be an $\GRB$-operator on a
bimodule $(V;\huaL,\huaR)$ over a pre-Lie algebra  $(\g,\cdot_\g)$ and $\Omega:\g\longrightarrow V_T$ a solution of the strong Maurer-Cartan equation on $\g\bowtie V_T$. Then
\begin{itemize}
\item[$\rm(i)$]$(T,N,S)$ is an $\GRBN$-structure on the bimodule $(V;\huaL,\huaR)$ over the pre-Lie algebra $(\g,\cdot_\g)$, where $N=T\circ\Omega$ and $S=\Omega\circ T$;
    \item[$\rm(ii)$]$(\Omega,S,N)$ is an $\GRBN$-structure on the
bimodule $(\g;\frkL^T,\frkR^T)$ over the pre-Lie algebra  $(V_T,\cdot^T)$, where $N=T\circ\Omega$ and
$S=\Omega\circ T$.
\end{itemize}
\end{thm}
\begin{proof}  (i) By Proposition \ref{pro:MC-Nijenhuis}, $N=T\circ\Omega$ is a Nijenhuis operator on the pre-Lie algebra $(\g,\cdot_\g)$. Now we prove that $(N,S)$ is a Nijenhuis structure on the bimodule $(V;\huaL,\huaR)$ over the pre-Lie algebra $(\g,\cdot_\g)$. Since $T$ is an $\GRB$-operator and $\Omega$ satisfies \eqref{eq:SMC-1}, we have
\begin{eqnarray}
\label{eq:MC-3}\Omega T(\huaL_{T(u)}v+\huaR_{T(v)}u)=\Omega(T(u)\cdot_\g T(v))
=\huaL_{T(u)}\Omega T(v)+\huaR_{T(v)}\Omega T(u),\quad\forall~u,v\in V.
\end{eqnarray}
Replacing $y$ by $T(v)$ in \eqref{eq:SMC-2}, by \eqref{eq:SMC-1}, we have
\begin{eqnarray*}
  0&=&\Omega(x)\cdot^T\Omega T(v)-\Omega\big(\frkL^T_{  \Omega(x)}T(v)+\frkR^T_{  \Omega T(v)}x\big)\\
  &=&\huaL_{T\Omega(x) }\Omega T(v)+\huaR_{T\Omega T(v)}\Omega(x)-\Omega\big(T\Omega(x)\cdot_\g T(v)-T(\huaR_{T(v)}\Omega(x))\big)\\
  &&-\Omega\big(x\cdot_\g T\Omega T(v)-T(\huaL_{x}\Omega T(v))\big)\\
  &=&\huaL_{T\Omega(x) }\Omega T(v)+\huaR_{T\Omega T(v)}\Omega(x)-\huaL_{T\circ \Omega(x)}\Omega T(v)-\huaR_{T(v)}\Omega T\Omega(x)\\
  &&+\Omega T(\huaR_{T(v)}\Omega(x))-\huaL_x\Omega T\Omega T(v)-\huaR_{T\circ \Omega T(v)}\Omega(x)+\Omega T(\huaL_{x}\Omega T(v))\\
  &=&-\huaR_{T(v)}\Omega T\Omega(x)+\Omega T(\huaR_{T(v)}\Omega(x))-\huaL_x\Omega T\Omega T(v)+\Omega T(\huaL_{x}\Omega T(v)),
\end{eqnarray*}
which implies that
\begin{equation}\label{eq:MC-4}
  \Omega T(\huaR_{T(v)}\Omega(x))-\huaR_{T(v)}\Omega T\Omega(x)=\huaL_x\Omega T\Omega T(v)-\Omega T(\huaL_{x}\Omega T(v)).
\end{equation}
By \eqref{eq:MC-3} and \eqref{eq:MC-4}, we obtain
\begin{equation*}
  \huaL_{T\Omega(x)}\Omega T(v)- \Omega T(\huaL_{T\Omega(x)}v)=\huaL_x\Omega T\Omega T(v)-\Omega T(\huaL_{x}\Omega T(v)).
\end{equation*}
This is just the condition \eqref{eq:coNijpair1} for $N=T\circ\Omega$ and $S=\Omega\circ T$.

Similarly, replacing $x$ by $T(u)$ in \eqref{eq:SMC-2}, by \eqref{eq:SMC-1}, we have
\begin{eqnarray*}
  \huaR_y\Omega T\Omega T(u)-\Omega T(\huaR_y\Omega T(u))=\Omega T(\huaL_{T(u)}\Omega(y))-\huaL_{T(u)}\Omega T\Omega(y).
\end{eqnarray*}
By \eqref{eq:MC-3}, we have
\begin{equation*}
  \huaR_{T\Omega(y)}\Omega T(u)- \Omega T(\huaR_{T\Omega(u)}y)=\huaR_y\Omega T\Omega T(u)-\Omega T(\huaR_{y}\Omega T(u)).
\end{equation*}
This is just the condition \eqref{eq:coNijpair2} for $N=T\circ\Omega$ and $S=\Omega\circ T$. Thus $(N=T\circ\Omega,S=\Omega\circ T)$ is a Nijenhuis structure on the bimodule $(V;\huaL,\huaR)$ over the pre-Lie algebra  $(\g,\cdot_\g)$.

It is obvious that $T\circ S=N\circ T=T\circ\Omega\circ T$.
By direct calculation, we have
\begin{eqnarray*}
 {u\cdot^{T}_{S} v}-{u\cdot^{T\circ S}v}&=&\huaL_{T(u)}S(v)+\huaR_{T(v)}S(u)-S\Big(\huaL_{T(u)}v+\huaR_{T(v)}u\Big)\\
 &=&\huaL_{T(u)}\Omega T(v)+\huaR_{T(v)}\Omega T(u)-\Omega T\Big(\huaL_{T(u)}v+\huaR_{T(v)}u\Big).
\end{eqnarray*}
\eqref{eq:MC-3} implies that
 $${u\cdot^{T}_{S} v}={u\cdot^{T\circ S}v}.$$
 Thus $(T,N,S)$ is an $\GRBN$-structure on the bimodule $(V;\huaL,\huaR)$ over the pre-Lie algebra $(\g,\cdot_\g)$.

  (ii) can be proved similarly. We omit the details.
  \end{proof}

Next we consider the converse of Theorem~\ref{thm:MC-GRBN}.
\begin{thm}\label{thm:GRBN-MC}
 Let $(T,N,S)$ be an $\GRBN$-structure on the bimodule $(V;\huaL,\huaR)$ over the pre-Lie algebra  $(\g,\cdot_\g)$. If  $T$ is invertible, then $\Omega:=T^{-1}\circ N=S\circ T^{-1}:\g\rightarrow V$  is a solution of the strong Maurer-Cartan equation on $\g\bowtie V_T$.
\end{thm}
\begin{proof} Since $N=T\circ \Omega$ is a Nijenhuis operator on the pre-Lie algebra  $(\g,\cdot_\g)$, we have
\begin{eqnarray}\label{eq:MC5}
  T \Omega(x)\cdot_\g  T \Omega(y)=T\Omega\big(T \Omega(x)\cdot_\g y+x\cdot_\g  T \Omega(y)-T\Omega(x\cdot_\g y)\big).
\end{eqnarray}
By direct calculations, we have
\begin{equation}\label{eq:MC6}
  \frkL^T_{ \Omega(x)}y+\frkR^T_{  \Omega(y)}x=T \Omega(x)\cdot_\g y+x\cdot_\g  T \Omega(y)-T \Omega(x\cdot_\g y).
\end{equation}
Since $T$ is an $\GRB$-operator, \eqref{eq:MC5} and \eqref{eq:MC6} imply that
\begin{eqnarray*}
  T(\Omega(x)\cdot^T\Omega(y))= T \Omega(x)\cdot_\g  T \Omega(y)=T \Omega\big( \frkL^T_{ \Omega(x)}y+\frkR^T_{  \Omega(y)}x\big).
\end{eqnarray*}
Thus \eqref{eq:SMC-2} holds because $T$ is invertible.

Since ${u\cdot^{T}_{S} v}={u\cdot^{T\circ S}v}$ with $S=\Omega \circ T$, we have
\begin{eqnarray*}
\Omega T\Big(\huaL_{T(u)}v+\huaR_{T(v)}u\Big)=\huaL_{T(u)}\Omega T(v)+\huaR_{T(v)}\Omega T(u).
\end{eqnarray*}
Thus we have
\begin{eqnarray*}
  \Omega(T(u)\cdot_\g T( v))=\Omega T\Big(\huaL_{T(u)}v+\huaR_{T(v)}u\Big)=\huaL_{T(u)}\Omega T(v)+\huaR_{T(v)}\Omega T(u).
\end{eqnarray*}
Let $x=T(u)$ and $y=T(v)$, we have
$\Omega(x\cdot_\g y)=\huaL_{x}\Omega(y)+\huaR_{y}\Omega(x).$
Thus \eqref{eq:SMC-1} follows from that $T$ is invertible. We finish the proof.\end{proof}

By Theorem \ref{pro:TS1} and Proposition \ref{pro:TS2}, we have
\begin{cor}
  Let $T: V\longrightarrow \g$ be an $\GRB$-operator  on a
bimodule $(V;\huaL,\huaR)$ over a pre-Lie algebra  $(\g,\cdot_\g)$ and $\Omega:\g\rightarrow V$ a solution of the strong Maurer-Cartan equation on $\g \bowtie V_T$.
Then $N \circ  T$ is also an $\GRB$-operator, where $N=T\circ\Omega$. Moreover, $T$ and $N\circ T$ are compatible.
\end{cor}

By Theorem \ref{thm:hierarchy}, we have
\begin{cor}\label{pro:MC-GRBN-cor3}
  Let $T: \g\longrightarrow A$ be an $\GRB$-operator  on a
bimodule $(V;\huaL,\huaR)$ over a pre-Lie algebra  $(\g,\cdot_\g)$. Let $\Omega:\g\rightarrow V$  be a Maurer-Cartan element on the  twilled pre-Lie algebra $ \g\bowtie V_T$. Then
\begin{itemize}
\item[$\rm(i)$] for all $k\in\Nat$, $T_k=(T\circ \Omega)^k\circ T$ are $\GRB$-operators on the bimodule $(V;\huaL,\huaR)$ over the pre-Lie algebra $(\g,\cdot_\g)$. Furthermore, for all $k,l\in\Nat$, $T_k$ and $T_l$ are compatible.
    \item[$\rm(ii)$]for all $k\in\Nat$, $\Omega_k=( \Omega\circ T)^k\circ \Omega$ are $\GRB$-operators on the bimodule $(\g;\frkL^T,\frkR^T)$ over the pre-Lie algebra  $(V_T,\cdot^T)$. Furthermore, for all $k,l\in\Nat$, $\Omega_k$ and $\Omega_l$ are compatible.
\end{itemize}
\end{cor}
\emptycomment{\begin{ex}
Consider the solution of the strong Maurer-Cartan equation on the  twilled algebra $\g_R\bowtie \g$ given by Example \ref{ex:RB-MC1}. Thus, by Theorem \ref{thm:MC-GRBN}, $(R,N,S)$ is an $\GRBN$-structure, where $N$ and $S$ are, respectively, given by
  \begin{eqnarray*}
    N(f)(x)=\int^x_0\kappa(t)f'(t)dt,\quad S(h)(x)=\kappa(x)h(x),\quad\forall~f\in \g,h\in \g.
  \end{eqnarray*}
Furthermore, by Proposition \ref{pro:S-Nijenhuis operator}, $S$ is a Nijenhuis operator on the pre-Lie algebra $(\g,\ast)$, where the multiplication $\ast:\g\otimes \g\longrightarrow \g$ is given by
\begin{equation*}
 ( f\ast g)(x)=g(x)\int^x_0f(t)dt+f(x)\int^x_0g(t)dt, \quad\forall~f,g\in \g.
\end{equation*}
By Corollary \ref{pro:MC-GRBN-cor3}, for each $i\in\Nat$, the operator $N^i\circ T$ given by
 $$(N^i\circ T)(f)(x)= \int^x_0\kappa^i(t)f(t)dt$$
is an $\GRB$-operator  and each pair \liu{of} these $\GRB$-operators \liu{is} compatible, i.e. for all $s_1,s_2\in\Real,i,j\in\Nat$,
$$(s_1N^i\circ T+s_2N^j\circ T)(f)(x)=s_1\int^x_0\kappa^i(t)f(t)dt+s_2\int^x_0\kappa^j(t)f(t)dt$$
are still $\GRB$-operators.
\end{ex}}

\begin{ex}
Consider the solution of the strong Maurer-Cartan equation on the  twilled pre-Lie algebra $\g\bowtie \g_R$ given by Example \ref{ex:RB-MC11}. Thus, by Theorem \ref{thm:MC-GRBN}, $(R,N,S)$ is an $\GRBN$-structure, where $N$ and $S$ are, respectively, given by
  \begin{eqnarray*}
    N(f)(x)=\lambda f(x),\quad S(g)(x)=\lambda g(x),\quad\forall~f,g\in \g.
  \end{eqnarray*}
\end{ex}

\begin{ex}
  Consider the solution of the strong Maurer-Cartan equation on the twilled pre-Lie algebra $\g\bowtie \g_R$ given by Example \ref{ex:RB-MC2}. Then $(R,N,S)$ are $\GRBN$-structures, where $R$, $N$ and $S$ are given as follows:
  \begin{itemize}
  \item[\rm(1)] $R =\begin{bmatrix}0&
0&0\\ r_{21}&r_{22}&0\\ r_{31}& r_{32}& 0\end{bmatrix},\quad N=0,\quad S=\begin{bmatrix}0&
0&0\\ 0&0&0\\ 0& r_{21}\omega_{32}& r_{22}\omega_{32}\end{bmatrix};$
   \item[\rm(2)] $R =\begin{bmatrix}0&
0&0\\ 0&0&0\\ r_{31}& r_{32}& r_{33}\end{bmatrix}(r_{33}\neq 0),\quad N=\begin{bmatrix}0&
0&0\\ 0&0&0\\ 0& r_{32}\omega_{32}& 0\end{bmatrix},\quad S=0;$
    \item[\rm(3)] $R =\begin{bmatrix}0&
0&0\\ r_{21}&0&0\\ r_{31}& r_{32}& 0\end{bmatrix},\quad N=\begin{bmatrix}0&
0&0\\ 0&0&0\\0& r_{32}\omega_{22}& 0\end{bmatrix},\quad S=\begin{bmatrix}0&
0&0\\ r_{21}\omega_{22}&0&0\\r_{31}\omega_{22}& r_{32}\omega_{22}& 0\end{bmatrix}$.
\end{itemize}
\end{ex}
\subsection{$\KVB$-structures and  solutions of the strong Maurer-Cartan equation}\label{sec:Application}

In this subsection, we introduce the notion of a $\KVB$-structure, which has a close relationship with solutions of the strong Maurer-Cartan equation, $\KVN$-structures and $\HN$-structures.
\begin{defi}
  Let $r\in\Sym^2(\g)$ be an $\frks$-matrix and $\frkB\in\Sym^2(\g^*)$ be $\dt$-closed on a pre-Lie algebra $(\g,\cdot_\g)$. Then $(r,B)$ is called a {\bf $\KVB$-structure} if $\frkB_N$ is also $\dt$-closed, where $N=r^\sharp\circ \frkB^\natural$ and $\frkB_N\in\Sym^2(\g^*)$ is defined by $\frkB_N(x,y)=\frkB(N(x),y)$.
\end{defi}

\begin{ex}\label{ex:KVB1}
Let $\g$ be the $3$-dimensional pre-Lie algebra given in Example \ref{ex:HN2}.
Then $(r,\frkB)$ given by
 \begin{eqnarray*}
r&=&r_{11} e_1\otimes e_1+r_{22} e_2\otimes e_2+r_{23} e_2\otimes e_3+r_{23} e_3\otimes e_2,\\
\frkB&=&a e_1^*\otimes e_1^*+ b e_2^*\otimes e_3^*+b e_3^*\otimes e_2^*+ce_3^*\otimes e_3^*
 \end{eqnarray*}
 is a $\KVB$-structure on the pre-Lie algebra $\g$, where $a,b,c,r_{11},r_{22}$ and $r_{23}$ are constants.
\end{ex}
\begin{ex}\label{ex:KVB2}
    Let $(\g,\cdot_\g)$ be a 3-dimensional pre-Lie algebra
with a basis $\{e_1,e_2,e_3\}$ whose non-zero products are given
as follows:
$$e_2\cdot_\g e_3=e_1,\quad e_3\cdot_\g e_2=-e_1.$$
  Let $\{e_1^*,e_2^*,e_3^*\}$ be the dual basis. Then $(r,\frkB)$ given by
\begin{eqnarray*}
r&=&r_{11} e_1\otimes e_1+r_{12}e_1\otimes e_2+r_{12}e_2\otimes e_1+r_{13} e_1\otimes e_3+r_{13} e_3\otimes e_1,\\
\frkB&=&a e_2^*\otimes e_2^*+ b e_2^*\otimes e_3^*+b e_3^*\otimes e_2^*+ce_3^*\otimes e_3^*
   \end{eqnarray*}
 is a $\KVN$-structure on the pre-Lie algebra $(\g,\cdot_\g)$, where $a,b,c,r_{11},r_{12}$ and $r_{13}$ are constants.
\end{ex}

\begin{thm}\label{thm:comptoMC}
 Let $r\in\Sym^2(\g)$ be an $\frks$-matrix on a pre-Lie algebra $(\g,\cdot_\g)$. Then  $(r,\frkB)$ is a $\KVB$-structure if and only if $\frkB^\natural:\g\rightarrow \g^*$ is a solution of the strong Maurer-Cartan equation on the twilled pre-Lie algebra $\g\bowtie\g^*_{r^\sharp}$.
\end{thm}
\begin{proof}
Assume that $(r,\frkB)$ is a $\KVB$-structure. Note that $\frkB$ is $\dt$-closed if and only if it satisfies
$$\frkB^\natural(x\cdot_\g y)=\ad^*_x \frkB^\natural(y)-R^*_y\frkB^\natural(x),$$
which is just \eqref{eq:SMC-1} with $\Omega=\frkB^\natural$.

On one hand, for \eqref{eq:SMC-2} with $\Omega=\frkB^\natural$ and $T=r^\sharp$, we have
\begin{eqnarray}
  \nonumber&&\langle\frkB^\natural(x)\cdot_{r^\sharp}  \frkB^\natural(y)-\frkB^\natural({\ad^{*{r^\sharp}}_{\frkB^\natural(x)}}y-R^{*r^\sharp}_{\frkB^\natural(y)}x),z\rangle\\
 \nonumber &=&\langle \ad^*_{r^\sharp\frkB^\natural(x)}\frkB^\natural(y)-R^*_{r^\sharp\frkB^\natural(y)}\frkB^\natural(x),z\rangle-\langle\frkB^\natural\big(r^\sharp(\xi)\cdot_\g x+r^\sharp(R^*_x \xi)+x\cdot_\g r^\sharp(\xi)-r^\sharp(\ad^*_x \xi)\big),z\rangle\\
 \label{eq:SMC-s-matrix} &=&\frkB(r^\sharp \frkB^\natural(y)\cdot_\g z,x)+\frkB(y\cdot_\g r^\sharp\frkB^\natural(z),x)-\frkB(r^\sharp \frkB^\natural(x)\cdot_\g z,y)-\frkB(x\cdot_\g r^\sharp \frkB^\natural(z),y)\\
 \nonumber &&-\frkB([r^\sharp \frkB^\natural(x),y]_\g,z)+\frkB([r^\sharp \frkB^\natural(y),x]_\g,z).
\end{eqnarray}
  On the other hand, letting $N=r^\sharp\circ \frkB^\natural$ and by the fact that $\frkB$ is $\dt$-closed, we have
\begin{eqnarray*}
  \frkB(y,N(x)\cdot_\g z)-\frkB(N(x),y\cdot_\g z)+\frkB([N(x),y]_\g,z)&=&0;\\
  \frkB(N(y),x\cdot_\g z)-\frkB(x,N(y)\cdot_\g z)+\frkB([x,N(y)]_\g,z)&=&0;\\
  \frkB(y,x\cdot_\g N(z))-\frkB(x,y\cdot_\g N(z))+\frkB([x,y]_\g,N(z))&=&0,
\end{eqnarray*}
which implies that
\begin{eqnarray*}
&&\frkB(y,N(x)\cdot_\g z)+\frkB(N(y),x\cdot_\g z)+\frkB(y,x\cdot_\g N(z))-\frkB(N(x),y\cdot_\g z)-\frkB(x,N(y)\cdot_\g z)\\
 &&-\frkB(x,y\cdot_\g N(z))+\frkB([N(x),y]_\g,z)+\frkB([x,N(y)]_\g,z)+\frkB([x,y]_\g,N(z))=0.
\end{eqnarray*}
Since $\frkB_N$ is also $\dt$-closed, we have
$$\frkB(N(y),x\cdot_\g z)-\frkB(N(x),y\cdot_\g z)+\frkB([x,y]_\g,N(z))=0.$$
Thus
\begin{eqnarray*}
 &&\frkB(y, x\cdot_\g N(z))-\frkB(x ,y\cdot_\g N(z))+\frkB(y ,N(x)\cdot_\g z)\\
 &&-\frkB(x, N(y)\cdot_\g z)+\frkB(z,[N(x),y]_\g)+\frkB(z,[x ,N(y)]_\g)=0,
\end{eqnarray*}
which is just \eqref{eq:SMC-s-matrix} with $N=r^\sharp\circ \frkB^\natural$. By Proposition \ref{pro:SMC-O-operator}, $\frkB^\natural:\g\rightarrow \g^*$ is a solution of the strong Maurer-Cartan equation on the twilled pre-Lie algebra $\g\bowtie\g^*_{r^\sharp}$.

The converse can be proved similarly.
\end{proof}

The following proposition demonstrates the relation between $\KVB$-structures and $\HN$-structures.
\begin{pro}\label{pro:ConNij-maxCon}
If $(r,\frkB)$ is a $\KVB$-structure on a pre-Lie algebra $(\g,\cdot_\g)$ and $\frkB$ is invertible, then $(\frkB,N=r^\sharp\circ \frkB^\natural)$ is an $\HN$-structure.

Conversely, if $(\frkB,N)$ is an $\HN$-structure, then $(r,\frkB)$ is a $\KVB$-structure, where $r\in\Sym^2(\g)$ is determined by $r^\sharp=N\circ (\frkB^\natural)^{-1}$.
\end{pro}
\begin{proof}
By Theorem \ref{thm:comptoMC} and Theorem \ref{thm:MC-GRBN}, if $(r,\frkB)$ is a $\KVB$-structure, then $N=r^\sharp\circ \frkB^\natural$ is a Nijenhuis operator on $(\g,\cdot_\g)$. By the definition of a $\KVB$-structure, it is obvious that   $(\frkB,N=r^\sharp\circ \frkB^\natural)$ is an $\HN$-structure.

Conversely, since $(\frkB,N)$ is an $\HN$-structure, by Corollary \ref{cor:Connes-r-matrix-Nij}, $(\tilde{r},N)$ is a $\KVN$-structure with $\tilde{r}\in\Sym^2(\g)$ given by \eqref{eq:Connes-r}, and thus $r$ determined by $r^\sharp=N\circ (\frkB^\natural)^{-1}$ is an $\frks$-matrix. It is obvious that $\frkB_N$ is $\dt$-closed. Thus $(r,\frkB)$ is a $\KVB$-structure.
\end{proof}

Moreover, by Proposition \ref{thm:rNijRB}, Theorem \ref{thm:MC-GRBN} and Theorem \ref{thm:comptoMC}, we have
\begin{cor}\label{cor:tttt}
If $(r,\frkB)$ is a $\KVB$-structure on a pre-Lie algebra $(\g,\cdot_\g)$, then $(r,N=r^\sharp\circ \frkB^\natural)$ is a $\KVN$-structure.
Furthermore, for all $k\in\Nat$, $r_k\in\Sym^2(\g)$ defined by $r_k(\xi,\eta)=\langle(r^\sharp\frkB^\natural)^k r^\sharp(\xi),\eta\rangle$ for all~$\xi, \eta \in \g^*$, are pairwise compatible  $\frks$-matrices.
\end{cor}
Conversely, by Theorem \ref{thm:GRBN-MC} and Theorem \ref{thm:comptoMC}, we have
\begin{pro}
 Let $(r,N)$ be a $\KVN$-structure on the pre-Lie algebra $(\g,\cdot_\g)$. If $r$ is nondegenerate, then $(r,\frkB)$ is a $\KVB$-structure on  $(\g,\cdot_\g)$, where $\frkB\in\Sym^2(\g^*)$ is given by $\frkB^\natural=(r^\sharp)^{-1}\circ N$.
\end{pro}

\begin{ex}
 Consider the $\KVB$-structure $(r,\frkB)$ given by  Example \ref{ex:KVB1}.  Then by Corollary \ref{cor:tttt}, $(r,N)$ given by $$r=r_{11} e_1\otimes e_1+r_{22} e_2\otimes e_2+r_{23} e_2\otimes e_3+r_{23} e_3\otimes e_2,\quad N =\begin{bmatrix}ar_{11}&0&
0\\ 0&br_{23}&br_{22}+cr_{23}\\0&0&br_{23}\end{bmatrix}$$
 is a $\KVN$-structure.
\end{ex}

\begin{ex}
 Consider the $\KVB$-structure $(r,\frkB)$ given by  Example \ref{ex:KVB2}.  Then by Corollary \ref{cor:tttt}, $(r,N)$ given by $$r=r_{11} e_1\otimes e_1+r_{12}e_1\otimes e_2+r_{12}e_2\otimes e_1+r_{13} e_1\otimes e_3+r_{13} e_3\otimes e_1,\quad N =\begin{bmatrix}0&ar_{12}+br_{13}&
br_{12}+cr_{13}\\ 0&0&0\\0&0&0\end{bmatrix}$$
 is a $\KVN$-structure.
\end{ex}

 \end{document}